\documentclass[final]{amsart}

\usepackage{amssymb}
\usepackage{amsmath}
\usepackage{mathtools}
\usepackage{microtype}
\usepackage{csquotes}
\usepackage[mathscr]{euscript}

\usepackage{hyperref}
\usepackage{amssymb}
\usepackage{latexsym}
\usepackage{wrapfig}

\usepackage{comment}

\usepackage{graphicx}
\usepackage{todonotes}

\usepackage{tikz}
\usetikzlibrary{matrix,arrows}

\usepackage[all]{xy}

\usepackage{setspace}

\numberwithin{equation}{section}

\theoremstyle{definition}
\newtheorem{definition}[equation]{Definition}
\newtheorem{remark}[equation]{Remark}
\newtheorem{example}[equation]{Example}
\newtheorem{que}[equation]{Question}

\newtheorem{constr}[equation]{Construction}

\newtheorem*{definition*}{Definition}
\newtheorem*{remark*}{Remark}
\newtheorem*{example*}{Example}
\newtheorem*{que*}{Question}
\newtheorem*{conj*}{Conjecture}
\newtheorem*{constr*}{Construction}

\theoremstyle{plain}
\newtheorem{lem}[equation]{Lemma}
\newtheorem{thm}[equation]{Theorem}
\newtheorem{prop}[equation]{Proposition}
\newtheorem{cor}[equation]{Corollary}

\newtheorem{thma}{Theorem}

\newtheorem{lema}[thma]{Lemma}
\newtheorem{propa}[thma]{Proposition}
\newtheorem{cora}[thma]{Corollary}

\newcommand{\pref}[2]{\hyperref[#2]{#1 \ref*{#2}}}
\hypersetup{urlcolor=blue, citecolor=blue, linkcolor=blue, colorlinks=true}

\newcommand{\bbF}{{\mathbb{F}}}

\newcommand{\bbH}{{\mathbb{H}}}

\newcommand{\bbN}{{\mathbb{N}}}

\newcommand{\bbP}{{\mathbb{P}}}

\newcommand{\bbR}{{\mathbb{R}}}

\newcommand{\bbZ}{{\mathbb{Z}}}

\newcommand{\calB}{{\mathcal{B}}}
\newcommand{\calC}{{\mathcal{C}}}
\newcommand{\calD}{{\mathcal{D}}}
\newcommand{\calE}{{\mathcal{E}}}

\newcommand{\calH}{{\mathcal{H}}}

\newcommand{\calP}{{\mathcal{P}}}

\newcommand{\calR}{{\mathcal{R}}}
\newcommand{\calS}{{\mathcal{S}}}

\newcommand{\calX}{{\mathcal{X}}}

\newcommand{\scrA}{{\mathscr{A}}}

\let\ORGvarepsilon=\varepsilon
\let\varepsilon=\epsilon
\let\epsilon=\ORGvarepsilon

\newcommand{\scal}{{\mathbf{scal}}}

\newcommand{\diff}{{\mathrm{Diff}}}

\newcommand{\Spin}{\mathrm{Spin}}
\newcommand{\SO}{\mathrm{SO}}
\newcommand{\ort}{\mathrm{O}}

\newcommand{\bord}{\calB ord}

\newcommand{\obj}[1]{{\mathrm{\textbf{obj}_{#1}}}}
\newcommand{\morph}[1]{{\mathrm{\textbf{mor}_{#1}}}}
\newcommand{\mor}[2]{{\mathrm{\textbf{mor}_{#1}}(#2)}}

\newcommand{\im}{\mathrm{im\ }}
\newcommand{\tr}{\mathrm{\textbf{tr }}}
\newcommand{\rk}[1]{\mathrm{{rank}(#1)}}

\newcommand{\emb}{{\mathrm{Emb}}}
\newcommand{\imm}{{\mathrm{Imm}}}
\newcommand{\mon}{{\mathrm{Mon}}}
\newcommand{\map}{{\mathrm{Map}}}

\newcommand{\aut}{\mathrm{Aut}}

\newcommand{\congarrow}{\overset{\cong}\longrightarrow}
\newcommand{\bun}{\mathrm{Bun}}

\newcommand{\hiso}{\mathrm{h}\mathbf{Iso}}
\newcommand{\haut}{\mathrm{h}\mathbf{Aut}}

\newcommand{\embeds}{\hookrightarrow}

\newcommand{\hp}[1]{\bbH\bbP^{#1}}

\newcommand{\too}{\longrightarrow}
\newcommand{\sign}{\mathrm{sign}}

\newcommand{\inddiff}{\mathrm{inddiff}}
\newcommand{\se}{\calS\calE}
\newcommand{\hotop}{\mathrm{hTop}}
\newcommand{\op}{\mathrm{op}}
\newcommand{\lst}{\mathrm{lst}}
\newcommand{\rst}{\mathrm{rst}}

\newcommand{\id}{\mathrm{id}}
\newcommand{\pt}{\mathrm{pt}}
\newcommand{\tor}{\mathrm{tor}}
\newcommand{\mtor}{\mathrm{mtor}}

\newcommand{\ie}{\mbox{i.\,e.\,}{}}

\title[The action of $\diff(M)$ on $\calR^+(M)$]{The action of the mapping class group on metrics of positive scalar curvature}

\author{Georg Frenck}
\thanks{G.F. was supported by the SFB 878 ``Groups, Geometry and Actions'' and by the Deutsche Forschungsgemeinschaft (DFG, German Research Foundation) under Germany 's Excellence Strategy – EXC 2044 – 390685587, Mathematics Münster: Dynamics – Geometry - Structure}
\address{KIT, Karlsruher Institut für Technologie\\
Englerstraße 2\\
76131 Karlsruhe}
\email{georg.frenck@kit.edu}
\email{math@frenck.net}
\urladdr{frenck.net/Math}

\date{\today}


\begin{document}

\begin{abstract}
We present a rigidity theorem for the action of the mapping class group $\pi_0(\diff(M))$ on the space $\calR^+(M)$ of metrics of positive scalar curvature for high dimensional manifolds $M$. This result is applicable to a great number of cases, for example to simply connected $6$-manifolds and high dimensional spheres. Our proof is fairly direct, using results from parametrised Morse theory, the $2$-index theorem and computations on certain metrics on the sphere. We also give a non-triviality criterion and a classification of the action for simply connected $7$-dimensional $\Spin$-manifolds.
\end{abstract}

\maketitle


\section{Introduction}

\subsection{Statement of the results}

\noindent For a closed manifold $M$ let $\calR^+(M)$ denote the space of all Riemannian metrics of positive scalar curvature on $M$. The diffeomorphism group $\diff(M)$ of $M$ acts on the space $\calR^+(M)$ by pullback and this action defines a group homomorphism
\[\Theta\colon\Gamma(M)\coloneqq\pi_0(\diff(M))\too\pi_0(\haut(\calR^+(M)))\]
from the \emph{mapping class group of $M$} to the group of homotopy classes of homotopy self-equivalences of $\calR^+(M)$. Our main result is that the image of this map is often very small. To state this precisely without too many technicalities, we confine ourselves to the special case where $M$ is simply connected and $\Spin$ in this introduction, but remark that we prove results for all manifolds of dimension at least $6$.

Let $\ell$ be a $\Spin$-structure on $M$ and recall that a \emph{$\Spin$-diffeomorphism} of $(M,\ell)$ is a pair $(f,\hat f)$ consisting of an orientation preserving diffeomorphism $f\colon M\to M$ and an isomorphism $\hat f\colon f^*\ell\to \ell$ of $\Spin$-structures. We denote by $\diff^\Spin(M,\ell)$ the group of all $\Spin$-diffeomorphisms and by $\Gamma^\Spin(M,\ell)\coloneqq\pi_0(\diff^\Spin(M,\ell))$ the \emph{$\Spin$-mapping class group} of $(M,\ell)$. For a diffeomorphism $f$ of $M$ denote the \emph{mapping torus} by $T_f\coloneqq M\times[0,1]/(x,0)\sim(f(x),1)$. If $(f,\hat f)$ is a $\Spin$-diffeomorphism, $T_f$ inherits a $\Spin$-structure. This construction defines a group homomorphism
\[T\colon\Gamma^\Spin(M,\ell)\too\Omega_d^\Spin\]
to the cobordism group of closed $d$-dimensional $\Spin$-manifolds. Our main result is the following.

\begin{thma}\label{thm:a}
	If $(M,\ell)$ is a simply connected $\Spin$-manifold of dimension $d-1\ge6$, there exists a group homomorphism 
	\[\se\colon\Omega_d^{\Spin}\too \pi_0(\haut(\calR^+(M))),\]
	such that the following diagram, where $F$ is the forgetful map, commutes
	\begin{center}
	\begin{tikzpicture}
		\node (0) at (0,1.5) {$\Gamma^\Spin(M,\ell)$};
		\node (1) at (0,0) {$\Gamma(M)$};
		\node (2) at (5,1.5) {$\Omega_d^\Spin$};
		\node (3) at (5,0) {$\pi_0(\haut(\calR^+(M))).$};
		
		\draw[->] (0) to node[auto]{$F$} (1);
		\draw[->] (0) to node[auto]{$T$} (2);
		\draw[->] (1) to node[auto]{$\Theta$} (3);
		\draw[->] (2) to node[auto]{$\se$} (3);
	\end{tikzpicture}
	\end{center}
\end{thma}
\noindent Note that \pref{Theorem}{thm:a} is true but vacuous for $\calR^+(M)=\emptyset$. Since $\Omega_7^\Spin=0$ (cf. \cite[Théorème II.16, p. 49]{thom_quelques}), $f^*\colon \calR^+(M)\to\calR^+(M)$ is homotopic to the identity for every $\Spin$-diffeomorphism $(f,\hat f)$ of a simply connected, $6$-dimensional $\Spin$-manifold $M$. Using computations in characteristic classes we get the following.

\begin{thma}\label{thm:actionofmcg}
	Let $M$ be a simply connected, stably parallelizable manifold of dimension $d-1\ge6$, equipped with a $\Spin$-structure. Let $(f,\hat f)$ be a $\Spin$-diffeomorphism. Then the map 
	\[f^*\colon\calR^+(M) \too\calR^+(M)\]
	is homotopic to the identity unless $d\equiv 1,2\; (\bmod\; 8)$. In the latter case, $(f^2)^*$ is homotopic to the identity. 
\end{thma}

\begin{remark*}
	Any orientation-preserving diffeomorphism $f$ of $M$ can be lifted to a $\Spin$-diffeomorphism if $M$ is simply connected. Therefore $f^*\colon \calR^+(M)\to\calR^+(M)$ is homotopic to the identity for each orientation preserving diffeomorphism of $M$ if $d\not\equiv 1,2\;(\bmod\; 8)$. This conclusion does not hold for orientation-reversing diffeomorphisms (for example it is false if $f\colon S^7\to S^7$ is an orthogonal matrix of determinant $-1$).
\end{remark*}

\noindent For more examples we refer to \cite[Chapter 4.1]{ownthesis}. Using \pref{Theorem}{thm:a} one can also use computational results on $\pi_0(\calR^+(M))$ and $\pi_0(\diff_{x_0}(M))$ (for example \cite{berw} and \cite{grw_abelian}) to find elements in $\pi_0$ and $\pi_1$ of the observer moduli space of psc-metrics for certain manifolds. In the situation of \pref{Theorem}{thm:a}, assume that $f^*g$ is homotopic to $g$ for one $g\in\calR^+(M)$. Then the mapping torus admits a psc metric and hence $\alpha(T_f)=0$. This has an interesting consequence for manifolds of dimension $7$. Recall that the map $\Omega_8^\Spin\congarrow \bbZ\oplus\bbZ$, $[W]\mapsto (\sign(W),\alpha(W))$ is an isomorphism. Since the signature of a mapping torus always vanishes we deduce

\begin{cora}\label{cor:a}
	Let $M$ be a simply connected $\Spin$-manifold of dimension $7$ and let $f$ be a $\Spin$-diffeomorphism. If there exists a metric $g\in\calR^+(M)$ such that $f^*g$ lies in the same path component as $g$, then the map $f^*\colon\calR^+(M)\to\calR^+(M)$ is homotopic to the identity.
\end{cora}

\begin{propa}\label{prop:spheres}
	Let $d\ge7$, let $f\colon S^{d-1}\to S^{d-1}$ be a $\Spin$-diffeomorphism and let $g_\circ$ denote the round metric. If $f^*g_\circ$  and $g_\circ$ lie in the same path component, then $f^*$ is homotopic to the identity. 
\end{propa}

\begin{remark*}
The first result concerning the action of the mapping class group on the space of positive scalar curvature metrics was given by Hitchin \cite{hitchin_spinors}, where he constructed a map $\inddiff\colon \pi_0(\calR^+(M^{d-1}))\times\pi_0(\calR^+(M^{d-1}))\too KO^{-d}(\pt)$ and used the Atiyah--Singer index theorem to show that $\inddiff(g,f^*g) = \alpha(T_f)$. Hence, the $\alpha$-invariant of the mapping torus of $f$ is an obstruction to $f$ acting trivially on $\pi_0(\calR^+(M))$. For $S^{d-1}$ with $d\ge9$ and $d\equiv1,2\;(\bmod\;8)$ there exist diffeomorphisms $f$ with $\alpha(T_f)\ne0$ which implies that $\calR^+(S^{d-1})$ is not connected in these dimensions. \pref{Theorem}{thm:a} shows that these are the only dimensions where simply connected, stably parallelizable manifolds admit such a diffeomorphism.
\end{remark*}

\begin{remark*}
In \cite{berw} a factorisation result similar to \pref{Theorem}{thm:a} is proven. It is shown that for certain manifolds the image of $\pi_0(\diff_\partial(M^{2n}))\to \pi_0(\haut\calR^+(M)))$ is abelian, where $\diff_\partial$ denotes those diffeomorphisms that fix a neighbourhood of the boundary point-wise. Using an obstruction theoretic argument they conclude that this map factors through $\pi_1(MT\Spin(2n))$. This has been upgraded in \cite{erw_psc2} and \cite{erw_psc3} to hold for a bigger class of manifolds. \pref{Theorem}{thm:a} directly implies abelianess of the image and improves the named results since the map $\pi_1(MT\Spin(d-1))\to\Omega_d^\Spin$ has nontrivial kernel.
\end{remark*}

\subsection{Outline of the proof} 
\pref{Theorem}{thm:a} follows from a more general, cobordism theoretic result which we will develop in this outline. The main geometric ingredient is a parametrised version of the famous Gromov--Lawson--Schoen--Yau surgery theorem due to Chernysh. Let $\varphi\colon S^{k-1}\times D^{d-k}\embeds M$ be an embedding and let $\calR^+(M,\varphi)\coloneqq\{g\in\calR^+(M)|\varphi^*g=g_\circ +g_{\tor}\}$ be the space of those metrics that have a fixed standard form on the image of $\varphi$.
\begin{thm}[{\cite[Theorem 1.1]{chernysh}, \cite[Main Theorem]{walsh_cobordism}}]
	If $d-k\ge3$, the inclusion $\calR^+(M,\varphi)\embeds \calR^+(M)$ is a weak equivalence.
\end{thm}
\noindent As a consequence we obtain a map 
\[\calS_\varphi\colon\calR^+(M)\dashrightarrow \calR^+(M,\varphi) \congarrow \calR^+(M_\varphi,\varphi^{\op})\embeds \calR^+(M_\varphi),\]
where the first map is the homotopy inverse to the inclusion and the second one is given by cutting out $\varphi_*(g_\circ + g_{\tor})$ and pasting in $\varphi^{\op}_*(g_{\tor} + g_\circ)$. Next we want to define the map $\calS$ for general cobordisms. In this paper, a \emph{cobordism} between $(d-1)$-dimensional manifolds $M_0$ and $M_1$ is a triple $(W,\psi_0,\psi_1)$ consisting of a $d$-dimensional manifold $W$ whose boundary has a decomposition $\partial W =\partial_0W\amalg \partial_1W$ and diffeomorphisms $\psi_i\colon \partial_iW\to M_i$ for $i=0,1$. We will only consider $\Spin$-structures on cobordisms in the final step of the proof. An \emph{admissible handle decomposition $H$ of $(W,\psi_0,\psi_1)$} is a collection of manifolds $N_1,\dots, N_{n}$, embeddings $\varphi_i\colon S^{k_i-1}\times D^{d-k_i}\embeds N_{i}$ with $d-k_i\ge3$ for $i=1,\dots,n$ and diffeomorphisms $f_0\colon \partial_0W\congarrow N_1$, $f_n\colon (N_n)_{\varphi_n}\congarrow \partial_1 W$ and $f_i\colon (N_{i})_{\varphi_{i}}\congarrow N_{i+1}$ for $i=1,\dots, n-1$ such that there exists a diffeomorphism $\mathrm{rel}\ \partial W$
\[W\cong \partial_0W\times[0,1]\cup_{f_0}\tr(\varphi_1) \cup_{f_1}\tr(\varphi_2)\cup_{f_2}\dots\cup_{f_{n-1}} \tr(\varphi_{n})\cup_{f_{n}}\partial_1W\times[0,1]\]
and $(W,\psi_0,\psi_1)$ is called an \emph{admissible cobordism} if it admits an admissible handle decomposition. By the theory of handle cancellation developed by Smale \cite{smale_structure} (see also \cite{kervaire_barden} and \cite{wall_connectivity1}), a cobordism is admissible if the inclusion $\psi_1^{-1}\colon M_1\embeds W$ is $2$-connected. For a cobordism $W$ with an admissible handle decomposition $H$ we define the surgery map $\calS_{W,H}\colon \calR^+(M_0)\to\calR^+(M_1)$ by
\[\calS_{W,H}\coloneqq (\psi_1)_*\circ (f_n)_*\circ \calS_{\varphi_{n}}\circ\dots\circ(f_1)_*\circ\calS_{\varphi_1}\circ (f_0)_*\circ(\psi_0)^*.\]
\begin{lema}
	Let $d\ge7$. Then the homotopy class of $\calS_{W,H}$ is independent of the choice of admissible handle decomposition $H$. We will write $\calS_W\coloneqq\calS_{W,H}$. If the inclusion $\psi_0^{-1}\colon M_0\embeds W$ is $2$-connected as well, $\calS_W$ is a weak homotopy equivalence.
\end{lema}
\begin{remark*}
	In \cite{walsh_parametrized2}, Walsh constructed a psc metric $g_H$ on $(W,H)$ that restricts to a given metric $g_0$ on $\partial_0W$. He shows that the homotopy class of $g_H$ is independent of $H$. Using boundary identifications $\psi_i$ this gives a well defined map $\calS_W\colon \pi_0(\calR^+(M_0))\to\pi_0(\calR^+(M_1))$. We adapt the proof from \cite{walsh_parametrized2} so that we obtain a well-defined homotopy class of a map of spaces inducing Walsh's map on $\pi_0$. 
\end{remark*}
\noindent To prove this one uses Cerf theory to show that different handle decompositions are related by a finite sequence of elementary moves. The parametrized handle exchange theorem of Igusa \cite{igusa_stability} ensures that these moves keep the handle decomposition admissible. Igusa's theorem is the point where $d\ge7$ is used. Next we show surgery invariance of $\calS_W$.

\begin{lema}
	Let $d\ge7$, let $M_0$, $M_1$ be two $(d-1)$-manifolds, let $W$ be an admissible cobordism and let $\Phi\colon S^{k-1}\times D^{d-k+1}\embeds \mathrm{Int }\ W$ be an embedding with $3\le k\le d-3$. Then $\calS_W \sim \calS_{W_\Phi}$.
\end{lema}

\noindent Now we are able to derive the general cobordism theoretic result. Let $\hat\Omega_d^\Spin$ denote the following category: objects are given by simply connected, $(d-1)$-dimensional $\Spin$-manifolds $M$ and morphisms from $M_0$ to $M_1$ are given by cobordism classes of $d$-dimensional $\Spin$-cobordisms $(W,\psi_0,\psi_1)$. Note that every such cobordism class contains an admissible cobordism and two admissible cobordisms in the same class are related by a sequence of surgeries satisfying the index constraints from the previous Lemma. 

\begin{thma}\label{thm:a-general}
	Let $d\ge7$. Then there exists a functor $\calS\colon\hat\Omega_d^\Spin\too \hotop$ into the homotopy category of spaces with the following properties:
	\begin{enumerate}
		\item On objects, $\calS$ is given by $\calS(M)=\calR^+(M)$,
		\item if $f\colon M_1\to M_0$ is a diffeomorphism, then $\calS(M_0\times[0,1],\id,f^{-1}) = f^*$,
		\item if $\alpha\in\hat\Omega_d^\Spin(M_0,M_1)$ is represented by $(\tr(\varphi), \id,\id)$ for $\tr(\varphi)$ the trace of a surgery datum $\varphi$ with codimension at least $3$, then $\calS(\alpha) = \calS_\varphi$.
	\end{enumerate}
	Furthermore, $\calS$ is uniquely determined by these properties, up to natural isomorphism.
\end{thma}

\noindent This immediately implies \pref{Theorem}{thm:a}: For a closed $\Spin$-manifold $V$ let $\se(V)\coloneqq \calS(M\times[0,1]\amalg V,\id,\id)$ and since $(M\times[0,1]\amalg T_f,\id,\id)$ is $\Spin$-cobordant to $(M\times[0,1],\id,f^{-1})$ the given diagram commutes.

\textbf{Acknowledgements.} This paper is a streamlined version of the author's PhD-thesis \cite{ownthesis} at WWU M\"unster. It is my great pleasure to thank Johannes Ebert for his guidance, lots of comments and many enlightening discussions. I also would like to thank Lukas Buggisch, Oliver Sommer and Rudolf Zeidler for many fruitful discussions.

\section{Handle decompositions and the surgery map}

\subsection{Spaces of Riemannian metrics}\label{sec:psc}
For a closed manifold $M$ we denote by $\calR(M)$ the contractible space of all Riemannian metrics on $M$ equipped with the (weak) Whitney $C^\infty$-topology. The subspace of metrics whose scalar curvature is strictly positive will be denoted by $\calR^+(M)$.

\begin{definition}\label{def:general-surgery-theorem}
	Let $M$ and $N$ be compact manifolds of dimension $d-1\ge0$ and let $\varphi\colon N\hookrightarrow M$ be an embedding. For a metric $g$ on $N$, we define
	\[\calR^+(M,\varphi;g):=\{h\in\calR^+(M)\colon \varphi^*h=g\}.\]
	For $N=\coprod_{i=1}^nS^{k_i-1}\times D^{d-k_i}$ and $g=\coprod_{i=1}^ng^{k_i-1}_\circ+g^{d-k_i}_{\tor}$ we write $\calR^+(M,\varphi):=\calR^+(M, \varphi;g)$. Here, $g^{k_i-1}_\circ$ denotes the round metric and $g^{d-k_i}_{\tor}$ a torpedo metric\footnote{A torpedo metric on $D^{d-k}$ is an $\ort(d-k)$-invariant metric of positive scalar curvature that restricts to the round metric on the boundary. For precise definitions see \cite{chernysh}, \cite{walsh_parametrized1} or \cite{ebertfrenck}.}. If there is no chance of confusion, we will omit the dimension of these metrics.
\end{definition}
		
			\noindent There is the following generalization of the famous Gromov--Lawson--Schoen--Yau surgery theorem (cf. \cite{gromovlawson_classification} and \cite{schoenyau}) which is originally due to Chernysh \cite{chernysh} and has been first published by Walsh \cite{walsh_cobordism}. A detailed exposition of Chernysh's proof can be found in \cite{ebertfrenck}. Let $M$ be a $(d-1)$-manifold and for $i=1,\dots, n$ let $N_i$ be closed manifolds of dimension $(k_i-1)$. Let $d-k_i\ge3$ for all $i$ and let $g_{N_i}$ be metrics on $N_i$ such that $\scal(g_{N_i} + g_{\tor})>0$. Let $N:=\coprod_{i=1}^nN_i\times D^{d-k_i}$, $g:=\coprod_{i=1}^ng_{N_i} + g_{\tor}$ and let $\varphi\colon N\embeds M$ be an embedding.

		\begin{thm}[Parametrized Surgery Theorem {\cite[Theorem 1.1]{chernysh}, \cite[Main Theorem]{walsh_cobordism}}]\label{thm:chernysh}
			The map
			\[\calR^+(M,\varphi; g)\hookrightarrow \calR^+(M)\]
			is a weak homotopy equivalence. In particular, if $M_1$ is obtained from $M_0$ by performing surgery along $\varphi\colon S^{k-1}\times D^{d-k}\embeds M_0$ of index $k \le d-3$ then there exists a zig-zag of maps 
			\[\calR^+(M_0)\overset\simeq\hookleftarrow\calR^+(M_0,\varphi) \overset\cong\longrightarrow \calR^+(M_1,\varphi^{\op})\hookrightarrow\calR^+(M_1).\]
			If furthermore $k\ge 3$, the rightmost map in this composition is also a weak equivalence and we obtain a zig-zag of weak equivalences from $\calR^+(M_0)$ to $\calR^+(M_1)$.
		\end{thm}
		
	\begin{remark}\label{rem:weak-equivalence}
		The space $\calR^+(M)$ is homotopy equivalent to a $CW$-complex (see \cite[Theorem 13]{palais_infinite-dimensional}). By Whitehead's theorem, a weak homotopy equivalence of $CW$-complexes is an actual homotopy equivalence. Therefore we may assume that weak homotopy equivalences of $\calR^+(M)$ have actual homotopy-inverses. 
	\end{remark}

\subsection{Handle decompositions of cobordisms}\label{sec:handle-decomposition}

In this section we discuss handle decompositions of a cobordism $W$. First, we give a model for attaching a handle. We adapt the one given in {\cite[Construction 8.1]{perlmutter_parametrized}} which is convenient. 

\begin{constr}[Standard trace]\label{con:standard-trace}
	Let $\epsilon\in(0,\frac14)$ be fixed and let $k\in\{0,\dots,d\}$. We fix once and for all an $\ort(k)\times\ort(d-k)$-invariant submanifold
	\[T_k\subset[0,1]\times D^k\times D^{d-k}\]
	with the following properties (see \pref{Figure}{fig:standard-trace} for a visualization)
	\begin{enumerate}
		\item $(s,0,0)\in T_k$ if and only if $s=\frac12$.
		\item The projection $T_k\overset{pr}\longrightarrow[0,1]$ is a Morse function and $(\frac12,0,0)$ is the only critical point of this Morse function. Its index is $k$. 
		\item We have the following equalities for intersections
			\begin{align*}
				T_k\cap([0,\epsilon)\times S^{k-1}\times D^{d-k}) &= [0,\epsilon)\times S^{k-1}\times D^{d-k}\\
				T_k\cap((1-\epsilon,1]\times D^{k}\times S^{d-k-1}) &= (1-\epsilon,1]\times D^{k}\times S^{d-k-1}\\
				T_k\cap([0,1]\times S^{k-1}\times S^{d-k-1}) &= [0,1]\times S^{k-1}\times S^{d-k-1} 
			\end{align*}
		\item The boundary of $T_k$ is given by 
			\[\partial T_k = (\{0\}\times S^{k-1}\times D^{d-k})\cup (\{1\}\times D^{k}\times S^{d-k-1})\cup([0,1]\times S^{k-1}\times S^{d-k-1}).\]
	\end{enumerate}
	We call $T_k$ the \emph{standard trace} of a $k$-surgery. 
	\begin{figure}[ht]
	\centering
	\includegraphics[width=15em]{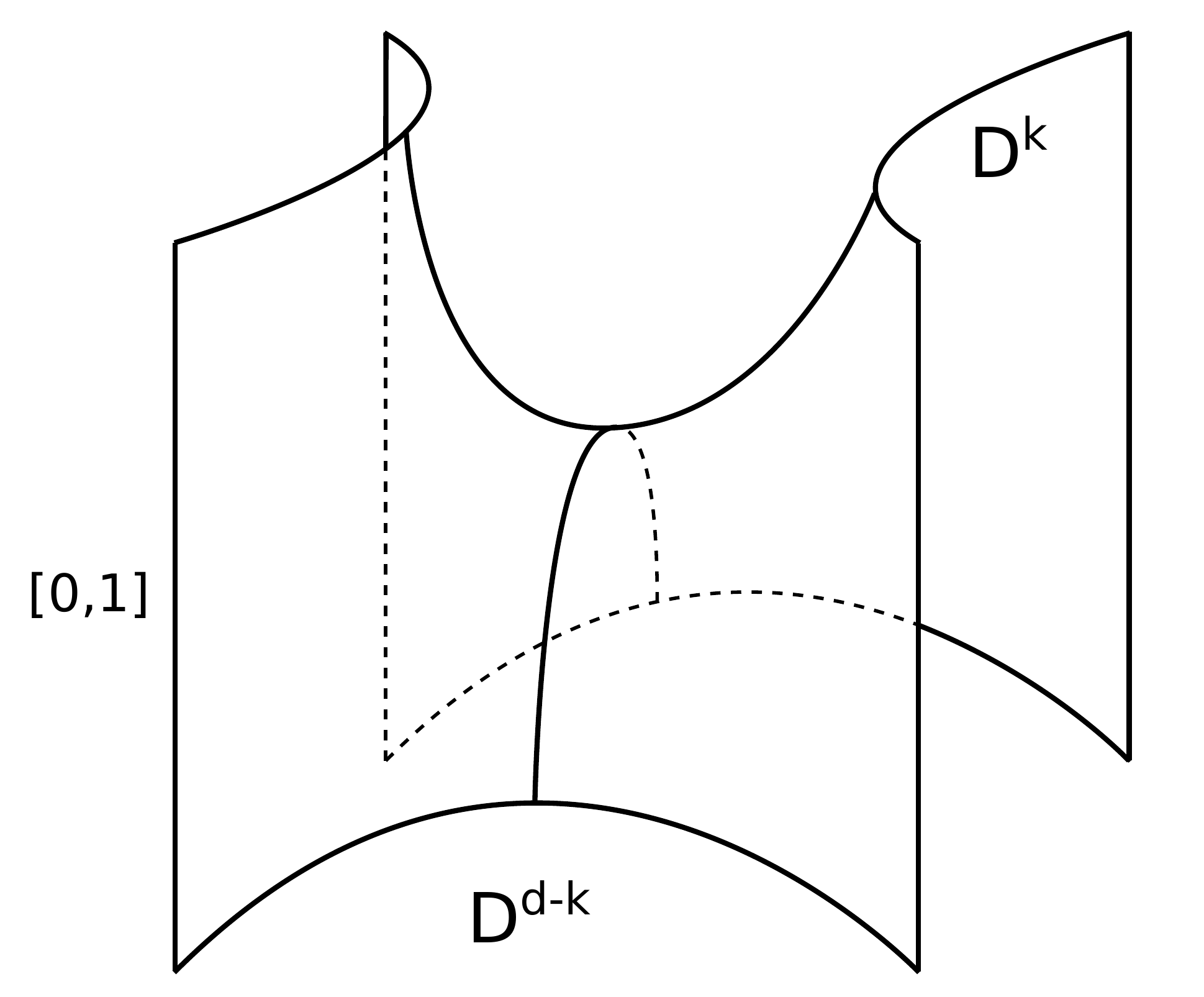}
	\caption{A standard trace.}\label{fig:standard-trace}
	\end{figure}
\end{constr}

\begin{definition}[Trace of a surgery]
	Let $M$ be a manifold and let $\varphi\colon S^{k-1}\times D^{d-k}\embeds M$ be an embedding. We call such an embedding a \emph{$k$-surgery datum in $M$} and we define the \emph{trace of $\varphi$} to be 	
	\[\tr(\varphi)\coloneqq \Bigl([0,1]\times (M\setminus \im\varphi)\Bigr) \cup_{\id_{[0,1]}\times \varphi} T_k.\]
	There is a Morse function $h_\varphi\colon \tr(\varphi)\to [0,1]$ with precisely one critical point with value $\frac12$ and index $k$. We define $M_\varphi:=h_\varphi^{-1}(1) \cong (M\setminus\im\varphi) \cup (D^k\times S^{d-k-1})$.
\end{definition}

\noindent For a surgery datum $\varphi$ in $M$ there is an obvious reversed surgery datum $\varphi^{\op}\colon S^{d-k-1}\times D^k\embeds M_\varphi$ and there is a canonical diffeomorphism $(M_\varphi)_{\varphi^{\op}}\cong M$. We define the \emph{attaching sphere of $\varphi$} to be $\varphi(S^{k-1}\times\{0\})\subset M$ and the \emph{belt sphere of $\varphi$} as $\varphi^{\op}(\{0\}\times S^{d-k-1})\subset M_\varphi$.

\begin{definition}\label{def:handle-decomposition}
	\begin{enumerate}
		\item Let $(W,\psi_0,\psi_1)\colon M_0\leadsto M_1$ be a cobordism and let $\varphi\colon S^{k-1}\times D^{d-k}\embeds M_1$ be an embedding. We define the manifold \emph{$W$ with a $k$-handle attached along $\varphi$} to be $(W\cup_{\psi_1} \tr(\varphi),\psi_0,\id)$. 
		\item A \emph{handle decomposition of $(W,\psi_0,\psi_1)\colon M_0\leadsto M_1$} is a collection of manifolds $N_1,\dots, N_{n}$, embeddings $\varphi_i\colon S^{k_i-1}\times D^{d-k_i}\embeds N_{i}$ for $i=1,\dots,n$ and diffeomorphisms $f_0\colon \partial_0W\congarrow N_1$, $f_n\colon (N_n)_{\varphi_n}\congarrow \partial_1W$ and $f_i\colon (N_{i})_{\varphi_{i}}\congarrow N_{i+1}$ for $i=1,\dots, n-1$ such that there exists a diffeomorphism $\mathrm{rel}\ \partial W$
	\[W\cong \partial_0W\times[0,1]\cup_{f_0}\tr(\varphi_1) \cup_{f_1}\tr(\varphi_2)\cup_{f_2}\dots\cup_{f_{n-1}} \tr(\varphi_{n})\cup_{f_{n}}\partial_1W\times[0,1]\]
	We call $f_i$ the \emph{identifying diffeomorphisms} and $\varphi_i$ the \emph{surgery data}.
	\end{enumerate}
\end{definition}

\begin{remark}\label{rem:diffeomorphism-and-surgery}
	For a diffeomorphism $f\colon M_0\congarrow M_1$ and a surgery datum $\varphi$ in $M_0$ there exists a canonical induced diffeomorphism $F\colon \tr\varphi\to\tr(f\circ\varphi)$ that restricts to $f$ on the incoming boundary and to a diffeomorphism $f_\varphi\colon (M_0)_\varphi\to (M_1)_{f\circ\varphi}$ such that $f_\varphi$ is equal to $f$ on $M_0\setminus\im\varphi$ and $f_\varphi\circ \varphi^{\op} = (f\circ\varphi)^{\op}$ on the outgoing boundary.
\end{remark}

\noindent In order to compare different handle decompositions of a manifold, we need to describe a model for handle cancellation. Let $W\colon M_0\leadsto M_1$ be a cobordism which has a handle decomposition with two handles\footnote{For ease of notation we assume that all boundary identifications are given by the identity.}: Let $\varphi\colon S^{k-1}\times D^{d-k}\embeds M_0$ and $\varphi'\colon S^k\times D^{d-k-1}\embeds (M_0)_\varphi$ be two surgery data such that the belt sphere of $\varphi$ and the attaching sphere of $\varphi'$ intersect transversely in a single point. By \cite[Theorem 5.4.3]{wall_difftop} there exists an embedding of a disk $D^{d-1}\cong D\subset M_0$ such that $\im\varphi\subset D$ and $\im\varphi'\subset D_{\varphi}$. Therefore it suffices to have a closer look at handle cancellation on the sphere. Let $M_0 =  D\cup D' = S^{d-1}$ where $D'$ is another disk. Let $a\in S^{d-k-1}$ and $b\in S^k$ such that $\varphi^\op(0,a)=\varphi'(b,0)$ is the unique intersection point. Since the belt sphere of $\varphi$ and the attaching sphere of $\varphi'$ intersect transversally here, there is a disc $S^{k}_+\subset S^k$ such that $\varphi'(S^{k}_+\times \{0\})=\varphi^{op}(D^k\times\{a\})$ after possibly changing the coordinates of $D$. Let $S^k_-\coloneqq\overline{S^k\setminus S^k_+}$. Then $\varphi'(S^k_-\times\{0\})\subset M\setminus\im\varphi$ (see \pref{Figure}{fig:solid-torus}). 

\begin{figure}[ht]
\centering
\includegraphics[width=0.95\textwidth]{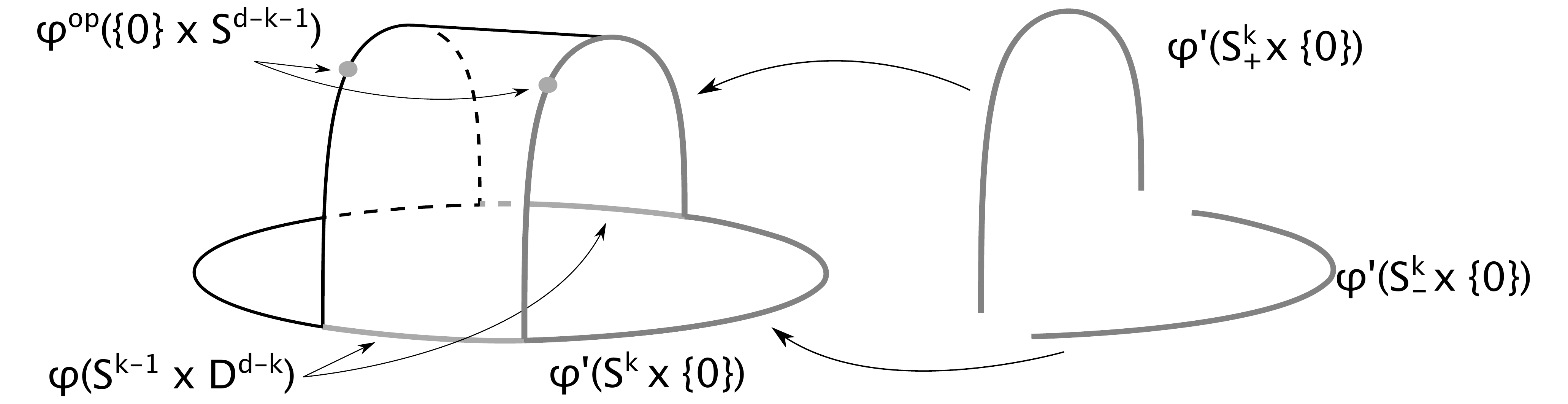}
\caption{}\label{fig:solid-torus}
\end{figure} 

\noindent Because of transversality we may isotopy $\varphi'$ such that $\varphi'(S^k_-\times D^{d-k-1})\subset M\setminus\im\varphi$. Then $\varphi(S^{k-1}\times D^{d-k})\cup\varphi'(S^k_-\times D^{d-k-1})\cong D^{d-1}$ (cf. \cite[Lemma 5.4.2.]{wall_difftop}) and also $A\coloneqq \overline{S^{d-1}\setminus \bigr(\varphi(S^{k-1}\times D^{d-k})\cup\varphi'(S^k_-\times D^{d-k-1})\bigr)}\cong D^{d-1}$. By choosing an identification $A\cong D^k\times D^{d-k-1}$ we have $\varphi'(S^k_-\times D^{d-k-1})\cup A \cong S^k\times D^{d-k-1}$. We see that
\[S^{d-1} = \bigr(\varphi(S^{k-1}\times D^{d-k})\cup\underbrace{\varphi'(S^k_-\times D^{d-k-1})\bigr)\cup A}_{\cong S^k\times D^{d-k-1}}\]
and hence we can change coordinates on $S^{d-1}$ by changing the embedding $D^{d-1}\embeds M$ such that $\varphi$ is the embedding of the first factor of the solid torus decomposition 
\[a^k\colon S^{d-1} \congarrow(S^{k-1}\times D^{d-k})\cup (D^k\times S^{d-k-1}),\] 
\ie $a_k\circ\varphi = \iota_{(S^{k-1}\times D^{d-k})}$. We get an induced map
{ \[a^k_\varphi\colon S^{d-1}_\varphi \congarrow (D^k\times S^{d-k-1})\cup (D^k\times S^{d-k-1}),\]} where we identify $(D^k\times S^{d-k-1})\cup (D^k\times S^{d-k-1}) = S^k\times S^{d-k-1} = (S^k\times D^{d-k-1})\cup (S^k\times D^{d-k-1})$. Because of transversality we may isotope $\varphi'$ so that $(a^k_\varphi)\circ \varphi'$ is equal to the inclusion of the first factor in $S^k\times D^{d-k-1}\cup S^k\times D^{d-k-1}$. Then
\[(a^k_\varphi)_{\varphi'}\colon (S^{d-1}_\varphi)_{\varphi'} \congarrow D^{k+1}\times S^{d-k-2}\cup S^k\times D^{d-k-1}\]
This is a solid torus decomposition of $(S^{d-1}_\varphi)_{\varphi'}$. We get a diffeomorphism $H_k\colon S^{d-1}\times [0,2]\congarrow\tr(\varphi)\cup\tr(\varphi')$ which fixes the strip $D'\times[0,2]\subset (S^{d-1})\times[0,2]$ and the lower boundary point-wise. We may also assume that $H_k$ restricts on the upper boundary to a diffeomorphism $\eta_k\colon S^{d-1}\congarrow (S^{d-1}_\varphi)_{\varphi'}$ which translates $(a^k_\varphi)_{\varphi'}$ into the solid torus decomposition $a^{k+1}$, \ie $\bigl((a^k_{\varphi})_{\varphi'}\circ\eta_k\bigr) = a^{k+1}$. For every $k\in\{0,\dots, d\}$ we fix the diffeomorphisms $H_k$ (and hence $\eta_k$) once and for all. The following proposition is well known and can be proven by analyzing paths of generalized Morse functions using Cerf theory (see \cite[Theorem 3.4]{connectedcerf} or \cite[Proposition 1.5.7]{ownthesis}).

	\begin{prop}\label{prop:handle-decomposition-difference}
	Let $d\ge7$. Then any two handle decompositions of $W$ only differ by a finite sequence of the following moves:
	\begin{enumerate}
		\item An identifying diffeomorphism is replaced by an isotopic one.
		\item A surgery datum is replaced by an isotopic one.
		\item A $k$-surgery datum is precomposed with an element $A\in \ort(k)\times\ort(d-k)$.
		\item The order of two surgery data with disjoint images is changed.
		\item Let  $\varphi$ and $\varphi'$ be $k$- and $(k+1)$-surgery data such that the belt sphere of $\varphi$ and the attaching sphere of $\varphi'$ intersect transversally in a single point. Then the two handles are replaced by the identifying diffeomorphism $\id\ \#\ \eta_k$.
	\end{enumerate}
	\end{prop}
	
\subsection{Hatcher-Igusa's 2-index theorem}
	Since \pref{Theorem}{thm:chernysh} has restrictions on the indices of surgery data, we need to consider handle decompositions with index constraints. Let $(W^d,\psi_0,\psi_1)\colon M_0\leadsto M_1$ be a cobordism. 
\begin{definition}
	$(W,\psi_0,\psi_1)$ is called \emph{admissible} if $\psi_1^{-1}\colon M_1\embeds W$ is $2$-connected. An \emph{admissible handle decomposition} is a handle decomposition where all surgery data $\varphi_i\colon S^{k_i-1}\times D^{d-k_i}\embeds N_{i}$ satisfy $k_i\le d-3$. 
\end{definition} 

\begin{remark}\label{rem:wall}
	It follows from the proof of the $h$-cobordism theorem due to Smale \cite{smale_structure} (see also \cite{kervaire_barden} and \cite{wall_connectivity1}) that every admissible cobordism admits an admissible handle decomposition. 
\end{remark}

	\noindent Next we want to analyze different admissible handle decompositions. Recall that a \emph{birth-death-singularity} of a function $f\colon W^d\to \bbR$ is a point $p\in W$ for which there exist coordinates $(x_1,\dots,x_d)$ around $p$ such that $f(x) = f(p) + x_1^3 + \sum_{i=2}^\lambda - x_i^2 + \sum_{i=\lambda+1}^d  x_i^2$ near $p$. In this case we call $(\lambda-1)$ \emph{the index of $f$ at $p$}. A function $f\colon W\to \bbR$ that has only non-degenerate and birth-death-singularities is called a \emph{generalized Morse function}.
	
	\begin{definition}
		We define $\calH(W)$ to be the space of generalized Morse functions on $W$. For $i\le j\in\{0,\dots,d\}$ we denote by $\calH_{i,j}(W)$ the space of generalized Morse functions such that non- degenerate critical points have index in $\{i,\dots, j\}$ and birth-death-singularities have index in $\{i,\dots,j-1\}$.
	\end{definition}

	\begin{thm}\label{thm:space-of-gmf-connected}
		Let $d\ge7$ and let $M_1\embeds W$ be $2$-connected. Then the space $\calH_{0,d-3}(W)$ is path-connected. If furthermore $M_0\embeds W$ is $2$-connected as well, the space $\calH_{3,d-3}(W)$ is path-connected, too. In particular, there exists a Morse-function without critical values of index $\{d-2,d-1,d\}$ or $\{0,1,2,d-2,d-1,d\}$ respectively.
	\end{thm}
	
	\noindent This follows from the \emph{parametrized handle exchange theorem}. It was first proven by Hatcher \cite{hatcher_higher} \enquote{in a short and elegant paper which ignores most technical details} \cite[p. 5]{igusa_stability}. A complete and rigorous proof has been given by Igusa in \cite{igusa_stability}. Note that there is an index shift: Igusa considers $n+1$-dimensional cobordisms, whereas our cobordisms are $d$-dimensional.

	\newtheorem*{paramthm}{Parametrized Handle Exchange Theorem}
	\begin{paramthm}[{\cite[p. 211, Theorem 1.1]{igusa_stability}}]\label{cor:hatcher-igusa}
	Let $i,j,k\in\bbN$ and assume that 
		\begin{enumerate}
		\setlength\itemsep{-0.2em}
			\item $(W,M_0)$ is $i$-connected,
			\item $j\ge i+2$,
			\item $i\le d-k-2-\min(j-1,k-1)$,
			\item $i\le d-k-4.$
		\end{enumerate}
		Then the inclusion $\calH_{i+1,j}(W)\hookrightarrow\calH_{i,j}(W)$ is $k$-connected. There is a dual version of this: Assume that
		\begin{enumerate}
		\setlength\itemsep{-0.2em}
			\item $(W,M_1)$ is $d-j$-connected,
			\item $j\ge i+2$,
			\item $d-j\le d-k-2-\min(j-1,k-1)$,
			\item $d-j\le d-k-4.$
		\end{enumerate}
		Then the inclusion $\calH_{i,j-1}(W)\hookrightarrow\calH_{i,j}(W)$ is $k$-connected.
	\end{paramthm}

	\begin{proof}[Proof of \pref{Theorem}{thm:space-of-gmf-connected}]
		Consider the chain of maps 
		\begin{align*}
			\calH_{3,d-3}(W)&\to\calH_{2,d-3}(W)\to\calH_{1,d-3}(W)\to\calH_{0,d-3}(W)\to\\
				&\to\calH_{0,d-2}(W)\to\calH_{0,d-1}(W)\to\calH(W)
		\end{align*}
		If $M_1\embeds W$ is $2$-connected and $d\ge7$, the last three maps are $1$-connected. If $M_0\embeds W$ is $2$-connected, the first three maps are $1$-connected as well. The theorem follows as $\calH(W)$ is connected.
	\end{proof}
	
\begin{remark}
	There is a small mistake in \cite[Proof of Theorem 3.1]{walsh_parametrized2}, where he only requires $d\ge6$. But the map $\calH_{0,d-2}(W)\embeds\calH(W)$ is only $0$-connected, \ie $\pi_0$-surjective but not necessarily $\pi_0$-injective under this assumption. Therefore it does not follow, that $\calH_{0,d-2}(W)$ is path-connected as claimed in loc. cit.. However, if $d\ge7$ the map is not only $\pi_0$-injective but also $1$-connected which is more than needed.
\end{remark}

\noindent The following result can again be proven by analyzing paths of generalized Morse functions with index constraints: Any two admissible handle decompositions arise from a Morse function having only critical points of index $\le d-3$. By \pref{Theorem}{thm:space-of-gmf-connected} there exists a path of generalized Morse functions also having only critical points of index $\le d-3$ and birth-death-points of index $\le d-4$. The rest of the proof is analogous to the one of \pref{Proposition}{prop:handle-decomposition-difference} (again, see \cite[Theorem 3.4]{connectedcerf} or \cite[Proposition 1.5.7 and Proposition 1.6.4]{ownthesis}).

	\begin{prop}[{\cite[1.6.4]{ownthesis}}]\label{prop:handle-decomposition-difference-ab}
	Let $W\colon M_0\leadsto M_1$ be an admissible cobordism of dimension $d\ge7$. Then any two admissible handle decompositions of $W$ only differ by a finite sequence of the 5 moves from \pref{Proposition}{prop:handle-decomposition-difference} with the following difference:
	\begin{enumerate}
		\item[5'.] Let $k\le d-4$ and let $\varphi$ and $\varphi'$ be $k$- and $(k+1)$-surgery data such that the belt sphere of $\varphi$ and the attaching sphere of $\varphi'$ intersect transversally in a single point. Then the two handles are replaced by the identifying diffeomorphism $\id\ \#\ \eta_k$.
	\end{enumerate}
	\end{prop}

\subsection{The surgery datum category}\label{sec:surgery-cat}

We recall the following method to construct a category. For details see \cite[pp. 48]{maclane_categories}.
\begin{definition}
	A \emph{graph} is a tupel $(O,A,\partial_0,\partial_1)$, where $O$ and $A$ are sets called \emph{object set} and \emph{arrow set} and $\partial_0,\partial_1$ are maps $A\rightrightarrows O$. We say that two arrows $f,g\in A$ are \emph{composable} if $\partial_0g = \partial_1f$.
\end{definition}

\begin{definition}
	Let $G=(O,A,\partial_0,\partial_1)$ be a graph. We define the category $\calC(G)$ to have elements of $O$ as objects and morphisms of $\calC(G)$ are (possibly empty) strings of composable morphisms of $A$. We call $\calC(G)$ the \emph{free category generated by $G$}.
\end{definition}

\begin{prop}[{\cite[p. 51, Proposition 1]{maclane_categories}}]\label{prop:quotient-category}
	Let $\calC$ be a small category and let $R$ be a binary relation, \ie a map that assigns to each pair $(a,b)$ of objects a subset of $\mor{\calC}{a,b}^2$. Then, there exists a category $\calC/R$ with object set $\obj\calC$ and a functor $Q\colon\calC\to\calC/R$ (which is the identity on objects) such that
	\begin{enumerate}
		\item If $(f,f')\in R(a,b)$ then $Qf=Qf'$.
		\item If $H\colon \calC\to \calD$ is a functor such that $(f,f')\in R(a,b)$ implies $Hf=Hf'$, then there exists a unique functor $H'\colon \calC/R\to \calD$ such that $H'\circ Q = H$.
	\end{enumerate}
\end{prop}

\noindent Let $\bord_d$ denote the category with objects $(d-1)$-manifolds and morphisms given by diffeomorphism classes of cobordisms $(W,\psi_0,\psi_1)$. The main goal of this chapter is to give a presentation of $\bord_d$, i.e. a graph $G$, a relation $R$ and an equivalence of categories $\calC(G)/R\congarrow\bord_d$. Let us first construct the graph $G$. Objects in $O$ are the objects of $\bord_d$ and arrows will be given by diffeomorphisms and elementary cobordisms:

	\begin{enumerate}
		\item For a diffeomorphism $f\colon M_0\to M_1$ we get an arrow $I_f\in A$ from $M_0$ to $M_1$.
		\item For a surgery datum $\varphi$ in $M$ we get an arrow $S_\varphi\in A$ from $M$ to $M_\varphi$.
	\end{enumerate}

\noindent Next, we need to construct the relation $R$ on $\calC(G)$. Recall that for a diffeomorphism $f\colon M\to M'$ and a surgery datum $\varphi$ in $M$ there exists a canonical induced diffeomorphism $f_\varphi\colon M_\varphi\to M'_{f\circ\varphi}$. Also, if $\varphi$ and $\varphi'$ are two surgery embeddings into $M$ with disjoint images, there are obvious induced surgery data $\varphi'_\varphi$ and $\varphi_{\varphi'}$ on $M_\varphi$ and $(M_{\varphi})_{\varphi'_\varphi} =  (M_{\varphi'})_{\varphi_{\varphi'}}$. We define $R$ to be the relation on morphism sets of $\calC(G)$ generated by the following:
	\begin{enumerate}
		\item $I_{\id} = \id$.
		\item If $f\colon M_0\congarrow M_1$ and $g\colon M_1\congarrow M_2$ are diffeomorphisms, then $I_g \circ I_f = I_{g\circ f}$.
		\item Let $f\colon M_0\congarrow M_1$ and let $\varphi$ be a surgery embedding into $M_0$. Then $S_{f\circ\varphi}\circ I_f = I_{f_\varphi} \circ S_{\varphi}$.
		\item If $f,f'\colon M\congarrow M'$ are isotopic, then $I_f=I_{f'}$.
		\item If $A\in O(k)\times O(d-k)$, then $S_\varphi = S_{\varphi\circ A}$.
		\item If $\varphi,\varphi'$ are two surgery embeddings into $M$ with disjoint images, then $S_{\varphi_{\varphi'}}\circ S_{\varphi'} = S_{\varphi'_\varphi}\circ S_\varphi$.
		\item Let $\varphi$ be a $k$-surgery datum in $M$ and $\varphi'$ a $(k+1)$-surgery datum in $M_\varphi$ such that the belt sphere of $\varphi$ and the attaching sphere of $\varphi'$ intersect transversely in a single point. Then $S_{\varphi'}\circ S_\varphi = I_{\id\ \#\ \eta_k}$, where $\eta_k$ is the diffeomorphism described \pref{Section}{sec:handle-decomposition}, below \pref{Remark}{rem:diffeomorphism-and-surgery}.
	\end{enumerate}
	
\begin{remark}
	For isotopic surgery embeddings $\varphi$ and $\varphi'$ we get a diffeotopy $H$ of $M$ such that $H_0=id$ and $H_1\circ\varphi = \varphi'$ by the isotopy extension theorem. Then 
	\[S_{\varphi'} = S_{H_1\circ\varphi} \circ I_{H_0} = S_{H_1\circ\varphi} \circ I_{H_1}  = I_{(H_1)_\varphi}\circ S_{\varphi}.\]	
\end{remark}

\begin{definition}\label{def:surgery-datum-category}
	We define the \emph{surgery datum category} $\calX_d$ to be $\calC(G)/R$ and $Q\colon \calC(G)\to\calX_d$ shall denote the projection functor.
\end{definition}

\subsection{A presentation of the cobordism category}\label{sec:presentation}

In this section we prove that the surgery datum gives a presentation of the category $\bord_d$. This is the main result of this chapter.

\begin{thm}\label{thm:presentation-of-bordism-category}
	Let $P\colon \calC(G)\to\bord_d$ denote the functor which is the identity on objects and is given on morphisms by
	\begin{enumerate}
		\item For $f\colon M_0\to M_1$, $I_f$ is mapped to $(M_0\times [0,1],\id,f)\cong (M_1\times [0,1],f^{-1},\id)$
		\item For a surgery datum $\varphi$ in $M$, $S_\varphi$ is mapped to $(\tr(\varphi),\id,\id)$.
	\end{enumerate}
	Then $P$ descends to a functor $\calP\colon\calX_d\to\bord_d$ which is an equivalence of categories. 
\end{thm}

\begin{proof}
	First we check well-definedness. By \pref{Proposition}{prop:quotient-category} it suffices to show that $P$ respects the relations of $\calX_d$.	
	\begin{enumerate}
		\item $(M_0\times [0,1],\id,\id)$ is the identity.
		\item$\begin{aligned}[t]
				(M_1\times [0,1],\id,f)&\circ (M_0\times [0,1],\id,g) \coloneqq (M_0\times [0,1]\cup_g M_1\times[0,1],\id,f)\\
					&\congarrow (M_0\times [0,2],\id,f\circ g)
			\end{aligned}$\\
			and the diffeomorphism is given by the identity on ${M_0\times[0,1]}$ and by the map $(p,t)\mapsto (g^{-1}(p),t+1)$ for $(p,t)\in M_1\times[0,1]$.
		\item Let $\varphi$ be a surgery embedding into $M_0$ and let $f\colon M_0\congarrow M_1$ be a diffeomorphism. 
			\begin{align*}
				P(I_{f_\varphi}\circ S_\varphi) & = (\tr(\varphi)\cup (M_0)_\varphi\times[0,1],\id, f_\varphi)\\
				P(S_{f\circ\varphi}\circ I_f) &=  ([0,1]\times M_0\cup_{f}\tr(f\circ\varphi),\id,\id)
			\end{align*}
			We will show that both of these are diffeomorphic to $X\coloneqq(M_0\times[0,1]\cup\tr\varphi\cup_{f_\varphi}(M_1)_{f\circ \varphi}\times[0,1],\id,\id)$. The diffeomorphism $X\congarrow P(I_{f_\varphi}\circ S_\varphi)$ is given by shrinking $M_0\times[0,1]\cup\tr\varphi$ to $\tr\varphi$ and by $f_\varphi\times\id$ on $(M_0)_{\varphi}\times[0,1]$. Recall that there is a canonical diffeomorphism $F\colon\tr\varphi\congarrow\tr(f\circ\varphi)$. The diffeomorphism $X\congarrow P(S_{f\circ\varphi}\circ I_f)$ is given by the identity on $M_0\times [0,1]$, F on $\tr(\varphi)$ and by shrinking the collar of $(M_1)_{f\circ\varphi}$.
		\item Let $f_t\colon M_0\congarrow M_1$ be a diffeotopy. Then we get a diffeomorphism $F\colon([0,1]\times M_0,\id,f_0)\congarrow ([0,1]\times M_0,\id,f_1)$ given by $F(t,x) = f_t^{-1}\circ f_0(x)$.
		\item For every $A\in \ort(k)\times \ort(d-k)$, $\varphi\circ A$ is just a reparametrization of $\varphi$ and hence this does not change $\tr(\varphi)$ since the standard model was chosen to be $\ort(k)\times \ort(d-k)$-invariant (cf. \pref{Construction}{con:standard-trace}).
		\item Let $\varphi,\varphi'$ be surgery embeddings into $M$ with disjoint images and let $U, U'$ be disjoint neighborhoods of $\im\varphi, \im\varphi'$ in $M$. Let $F\colon [0,2]\times M\congarrow [0,2]\times M$ be a diffeomorphism such that 
			\begin{enumerate}
				\item $F|_{[0,\frac\epsilon2)\times M\cup (2-\frac\epsilon2,2]\times M} =\id$
				\item $F(t,x) = (t+1,x)$ for $1-\epsilon_1>t>\epsilon_1$ and $x\in U$
				\item $F(t,x) = (t-1,x)$ for $2-\epsilon_{1}>t>1+\epsilon_1$ and $x\in U'$
			\end{enumerate} 
		Then, $F$ induces a diffeomorphism $\overline F\colon\tr(\varphi)\cup\tr(\varphi'_\varphi)\cong\tr(\varphi')\cup\tr(\varphi_{\varphi'})$ which is the identity on a collar of the boundary.
		\item This is precisely the situation discussed below \pref{Remark}{rem:diffeomorphism-and-surgery}.
	\end{enumerate}
	Therefore there is an essentially surjective functor $\calP\colon \calX_d\to\bord_d$. Every cobordism admits a handle decomposition and hence this functor is full. It is faithful by \pref{Proposition}{prop:handle-decomposition-difference}: Any two preimages of a cobordism $W$ under $\calP$ only differ by a finite sequence of the seven relations of $\calX_d$.	
\end{proof}

	\begin{definition}
		Let $a,b\in\{-1,0,1,\dots\}$. We define:
		\begin{enumerate}
			\item We define $\bord_d^{a,b}\subset \bord_d$ to be the wide\footnote{A subcategory is called wide if it contains all objects.} subcategory defined by the following: $\mor{\bord^{a.b}_d}{M_0,M_1}$ contains those morphisms $(W,\psi_0,\psi_1)$ where $\psi_0^{-1}\colon M_0\embeds W$ is $a$-connected and $\psi_1^{-1}\colon M_1{\embeds} W$ is $b$-connected. Here $(-1)$-connected shall be the empty condition.
			\item $G^{a,b}$ to be the graph with the same object set as $G$ and morphisms as follows: For $f\colon M_0\congarrow M_1$ we have $I_f\in A$ connecting $M_0$ and $M_1$ and for every surgery embedding $\varphi\colon S^{k-1}\times D^{d-k}\embeds M$ with $k\in[a+1,d-b-1]$ we have $S_\varphi\in A$ connecting $M$ and $M_\varphi$. Analogously to \pref{Definition}{def:surgery-datum-category}, we define $\calX_d^{a,b}\coloneqq \calC(G^{a,b})/R$.
		\end{enumerate}
	\end{definition}
	
	\noindent Note that $\bord_d^{a,b}$ is a category by the Blakers-Massey excision theorem \cite[Theorem 6.4.1]{tomdieck}.

	\begin{thm}\label{thm:presentation-of-bordism-category-ab}
	 	For $d\ge 7$, the induced functor $\calP^{-1,2}\colon\calX_d^{-1,2} \to \bord_d^{-1,2}$ is an equivalence of categories.
	\end{thm}

	\begin{proof}
		The proof goes along the same lines as the proof of \pref{Theorem}{thm:presentation-of-bordism-category}. For fullness we note that if the inclusions $\psi_1^{-1}\colon M_1\embeds W$ is $2$-connected respectively, there exists a Morse function with all indices $\le d-3$ by \pref{Theorem}{thm:space-of-gmf-connected}. Faithfulness follows from \pref{Proposition}{prop:handle-decomposition-difference-ab}.
	\end{proof}
	
\subsection{Definition of the surgery map}\label{sec:S-is-well-defined}
	Let $\hotop$ denote the homotopy category of spaces, \ie the category with spaces as objects and homotopy classes of maps as morphisms.	
		\begin{definition}\label{def:surgeryequivalence}
			We define a functor 
			\[\overline\calS\colon\calC(G^{-1,2})\too \hotop\]
			by the following:
			\begin{enumerate}
				\item $\overline\calS(M) = \calR^+(M)$.
				\item For a diffeomorphism $f\colon M_0\congarrow M_1$ the morphism $I_f$ is mapped to $[g\mapsto f_*g]$, where $f_*\coloneqq(f^{-1})^*$.
				\item For $\varphi\colon S^{k-1}\times D^{d-k}\embeds M$ with $k\le d-3$, 
				\[S_\varphi\mapsto [\calR^+(M)\dashrightarrow \calR^+(M,\varphi) \congarrow \calR^+(M_\varphi,\varphi^{\op})\embeds \calR^+(M_\varphi)],\]
				where the first map in this chain is the homotopy inverse to the inclusion (cf. \pref{Theorem}{thm:chernysh}) and the second one works as follows: For a metric $\tilde g$ on $M\setminus\im\varphi$, the metric $\tilde g \cup \varphi_*(g^{k-1}_\circ + g^{d-k}_{\tor})$ is mapped to $\tilde g \cup (\varphi^{\op})_*(g^k_{\tor} + g^{d-k-1}_\circ)$.
			\end{enumerate}
			We will abbreviate $\overline\calS_f\coloneqq\overline\calS(I_f)$ and $\overline\calS_\varphi\coloneqq \overline\calS(S_\varphi)$.
		\end{definition}
		
		\begin{remark}
			We have $\overline\calS(\mor{\calC(G^{2,2})}{M_0,M_1})\subset \hiso(\calR^+(M_0),\calR^+(M_1))$, i.e. $\overline\calS$ maps morphisms in ${\calC(G^{2,2})}$ to (the homotopy classes of) homotopy equivalences. This follows from the Parametrized Surgery Theorem (cf. \pref{Theorem}{thm:chernysh}).
		\end{remark}

\begin{lem}\label{lem:S-is-well-defined}
	$\overline\calS$ induces a well-defined functor $\calX_d^{-1,2}\too\hotop$.
\end{lem}

\begin{proof}
	For $d\le2$ the statement and the proof of this theorem is trivial since $\morph{\calX^{-1,2}_d}$ is generated by diffeomorphisms and it suffices to note that isotopic diffeomorphisms induce homotopic maps. Therefore we may assume $d\ge3$ throughout this proof. Throughout this proof we will use dashed arrows for maps that contain inverses of weak homotopy equivalences (cf. \pref{Remark}{rem:weak-equivalence}).
	
	We need to show that the relations $R$ from \pref{Definition}{def:surgery-datum-category} do not change the homotopy class of $\overline\calS(\alpha)$ for $\alpha\in\mor{\calX_d^{-1,2}}{M_0,M_1}$. This is obvious for relations $1,2$ and $4$. For relation $5$ this is easy as well, because $g_\circ + g_{\tor}$ is $O(k)\times O(d-k)$-invariant. Also, $S_{f\circ\varphi}\circ I_f$ and $I_{f_\varphi}\circ S_\varphi$ give homotopic maps because of the following homotopy-commutative diagram.
	\begin{center}
	\begin{tikzpicture}
		\node (0) at (0,1.2) {$\calR^+(M_0)$};
		\node (1) at (2.6,1.2) {$\calR^+(M_0,\varphi)$};
		\node (2) at (6.3,1.2) {$\calR^+((M_0)_\varphi,\varphi^{\op})$};
		\node (3) at (10,1.2) {$\calR^+((M_0)_\varphi)$};
		\node (4) at (0,0) {$\calR^+(M_1)$};
		\node (5) at (2.6,0) {$\calR^+(M_1,f\circ\varphi)$};
		\node (6) at (6.3,0) {$\calR^+((M_1)_{f\circ\varphi},(f\circ\varphi)^{\op})$};
		\node (7) at (10,0) {$\calR^+((M_1)_{f\circ\varphi})$};

		\draw[left hook->] (1) to (0);
		\draw[->,dashed,bend left=30] (0) to (1);
		\draw[->] (1) to (2);
		\draw[right hook->] (2) to (3);
		\draw[->] (1) to node[left]{$f_*$} (5);
		\draw[->] (3) to node[left]{$(f_{\varphi})_*$} (7);
		\draw[->] (0) to node[left]{$f_*$} (4);
		\draw[left hook->] (5) to (4);
		\draw[->] (2) to node[left]{$(f_{\varphi})_*$} (6);
		\draw[->,dashed,bend left=30] (4) to (5);
		\draw[->] (5) to (6);
		\draw[right hook->] (6) to (7);
	\end{tikzpicture}
	\end{center}
	
	\noindent For relation $6$ let $\varphi,\varphi'$ be two surgery embeddings into $M$ with disjoint images. Then there are inclusions $\calR^+(M,\varphi)\hookleftarrow\calR^+(M,\varphi\amalg\varphi')\embeds \calR^+(M,\varphi')$ and performing both surgeries at the same time is the same as performing them one after another. 

	The most difficult part of this proof is to show that handle cancellation does not alter the homotopy class of $\overline\calS(\alpha)$. If $d=3$ the only surgery data in $\morph{\calX_d^{-1,2}}$ are of the form $S^{-1}\times D^{3}\embeds M$. Hence there cannot be cancelling surgeries and we may assume that $d\ge4$ from now on. The proof now consists of two steps: We first reduce to the statement that cancelling surgeries do not change the path component of the round metric in $\calR^+(S^{d-1})$ which afterwards is proven by an elementary computation using \cite[Lemma 1.9]{walsh_parametrized1}.
	
	 Let $\varphi$, $\varphi'$ be surgery data in $M$ as in relation $7$ and let $f\coloneqq \id_M\ \#\ \eta_k$ where $\eta_k\colon S^{d-1}\congarrow (S^{d-1}_\varphi)_{\varphi'}$ is the fixed diffeomorphism from \pref{Section}{sec:handle-decomposition}. Note that in this case we have $k\le d-4$ and $d\ge4$. There exists an embedding of a disk $D^{d-1}\cong D\subset M$ such that $\im\varphi\subset D$ and $\im\varphi'\subset D_\varphi$. It suffices to show that the composition 
	\begin{center}
	\begin{tikzpicture}
		\node (0) at (0,1) {$\calR^+(M,D;g_{\tor})$};
		\node (1) at (2.75,1) {$\calR^+(M)$};
		\node (2) at (6.75,1) {$\calR^+\bigl((M_{\varphi})_{\varphi'}\bigr)$};
		\node (3) at (9.5,1) {$\calR^+(M)$};
		
		\draw[right hook->] (0) to node[above]{$\iota$} (1);
		\draw[->, dashed] (1) to node[above]{${\overline\calS_{\varphi'}\circ\overline\calS_\varphi}$}(2);
		\draw[->] (2) to node[above]{$f^*$} (3);
	\end{tikzpicture}
	\end{center}
	is homotopic to the inclusion $\iota$: By the \pref{Theorem}{thm:chernysh}, the inclusion map $\iota$ is a weak homotopy equivalence since $d\ge4$ and hence $\overline\calS_{\varphi'}\circ\overline\calS_{\varphi}$ is homotopic to $f_*$.
	
	Let $g\in\calR^+(D,\varphi)_{g_\circ}$ be a metric in the component of $g_{\tor}\in\calR^+(D)_{g_\circ}$ which exists by \pref{Theorem}{thm:chernysh}. Consider the following diagram:
	\begin{center}
	\begin{tikzpicture}
		\node (0) at (0,1.2) {$\calR^+(M\setminus D)_{g_\circ}$};
		\node (1) at (5,1.2) {$\calR^+(M,D;g)$};
		\node (2) at (3,0) {$\calR^+(M,D;g_{\tor})$};
		\node (3) at (8,1.2) {$\calR^+(M,\varphi)$};
		\node (4) at (8,0) {$\calR^+(M)$};
		
		\draw[->] (0) to node[above,sloped]{$\cong$} (1);
		\draw[->] (0) to node[above,sloped]{$\cong$} (2);
		\draw[->] (1) to node[above,sloped]{} (3);
		\draw[right hook->] (2) to node[above,sloped]{$\simeq$} (4);
		\draw[right hook->] (3) to node[right]{$\simeq$} (4);
	\end{tikzpicture}
	\end{center}
	The composition of the top maps is given by gluing in $g$ and the composition of the lower maps is given by gluing in $g_{\tor}$. These two metrics are homotopic relative to the boundary and hence this diagram commutes up to homotopy. The bottom map and the right-hand vertical map are weak equivalences by \pref{Theorem}{thm:chernysh} because $d\ge4$ and $k\le d-4$. Hence, the inclusion map $\calR^+(M,D;g)\embeds \calR^+(M,\varphi)$ is a weak equivalence as well. Let $g_\varphi$ be the metric obtained from $g$ by cutting out $\varphi_*(g_\circ^{k-1} + g^{d-k}_{\tor})$ and gluing in $\varphi^{\op}_*(g_{\tor}^{k} + g_\circ^{d-k-1})$. The following diagram where the horizontal maps are given by replacing $g$ with $g_\varphi$ commutes on the nose with the non-dashed arrows and up to homotopy with the dashed arrow:
	\begin{center}
	\begin{tikzpicture}
		\node (0) at (-2,0) {$\calR^+(M)$};
		\node (1) at (1,0) {$\calR^+(M,\varphi)$};
		\node (2) at (4.5,0) {$\calR^+(M_\varphi,\varphi^\op)$};
		\node (3) at (7.5,0) {$\calR^+(M_\varphi)$};
	
		\node (4) at (1,1.2) {$\calR^+(M, D;g)$};
		\node (5) at (4.5,1.2) {$\calR^+(M_\varphi, D_\varphi; g_\varphi)$};

		\draw[left hook->] (1) to node[above,sloped]{$\simeq$} (0);
		\draw[->, dashed, bend right=30] (0) to node[above,sloped]{} (1);
		\draw[left hook->] (4) to node[above,sloped]{$\simeq$} (0);
		\draw[->] (1) to node[above,sloped]{$\cong$} (2);
		\draw[right hook->] (2) to node[above,sloped]{$\simeq$} (3);
		\draw[right hook->] (4) to node[above,sloped]{$\simeq$} (1);
		\draw[->] (4) to node[above,sloped]{$\cong$} (5);
		\draw[right hook->] (5) to node[above,sloped]{} (2);
		\draw[right hook->] (5) to node[above,sloped]{} (3);
	\end{tikzpicture}
	\end{center}
	It again follows that the right-hand vertical map and the right-hand diagonal map are weak equivalences. Note that the composition of the bottom horizontal maps is precisely the map $\overline\calS_\varphi$. Now let $\tilde g\in\calR^+(D_\varphi,\varphi')_{g_\circ}$ be a metric in the component of $g_\varphi\in\calR^+(D_\varphi)_{g_\circ}$. We get the following diagram
	\begin{center}
	\begin{tikzpicture}
		\node (0) at (0,1.2) {$\calR^+(M_\varphi\setminus D_\varphi)_{g_\circ}$};
		\node (1) at (5,1.2) {$\calR^+(M_\varphi,D_\varphi;\tilde g)$};
		\node (2) at (3,0) {$\calR^+(M_\varphi,D_\varphi;g_\varphi)$};
		\node (3) at (8,1.2) {$\calR^+(M_\varphi,\varphi')$};
		\node (4) at (8,0) {$\calR^+(M_\varphi)$};
		
		\draw[->] (0) to node[above,sloped]{$\cong$} (1);
		\draw[->] (0) to node[above,sloped]{$\cong$} (2);
		\draw[->] (1) to node[above,sloped]{} (3);
		\draw[right hook->] (2) to node[above,sloped]{$\simeq$} (4);
		\draw[right hook->] (3) to node[right]{$\simeq$} (4);
	\end{tikzpicture}
	\end{center} 
	which is homotopy-commutative as $\tilde g$ and $g_\varphi$ are homotopic. The righthand vertical map is a weak equivalence because $d-k-1\ge3$ and we deduce that $\calR^+(M_\varphi,D_\varphi; \tilde g)\embeds \calR^+(M_\varphi, \varphi')$ is a weak equivalence as well. Let $\tilde g_{\varphi'}$ be the metric obtained from $\tilde g$ by cutting out $\varphi'_*(g_\circ^{k} + g^{d-k-1}_{\tor})$ and gluing in $\varphi'^{\op}{}_*(g_{\tor}^{k+1} + g_\circ^{d-k-2})$. We get the analogous homotopy-commutative diagram:

	\begin{center}
	\begin{tikzpicture}
		\node (0) at (-2,0) {$\calR^+(M_\varphi)$};
		\node (1) at (1.5,0) {$\calR^+(M_\varphi,\varphi')$};
		\node (2) at (6,0) {$\calR^+((M_\varphi)_{\varphi'},\varphi'^\op)$};
		\node (3) at (9.5,0) {$\calR^+((M_\varphi)_{\varphi'})$};
	
		\node (4) at (1.5,1.2) {$\calR^+(M_\varphi, D_\varphi;\tilde g)$};
		\node (5) at (6,1.2) {$\calR^+((M_\varphi)_{\varphi'}, (D_\varphi)_{\varphi'}; \tilde g_{\varphi'})$};
		
		\draw[left hook->] (1) to node[above,sloped]{$\simeq$} (0);
		\draw[->, dashed, bend right=30] (0) to node[above,sloped]{} (1);
		\draw[left hook->] (4) to node[above,sloped]{} (0);
		\draw[->] (1) to node[above,sloped]{$\cong$} (2);
		\draw[right hook->] (2) to node[above,sloped]{$\simeq$} (3);
		\draw[right hook->] (4) to node[above,sloped]{$\simeq$} (1);
		\draw[right hook->] (4) to node[above,sloped]{$\cong$} (5);
		\draw[right hook->] (5) to node[above,sloped]{} (2);
		\draw[->] (5) to node[above,sloped]{} (3);
	\end{tikzpicture}
	\end{center}
	This accumulates to the following diagram where all arrows are weak equivalences:
	\begin{center}
	\begin{tikzpicture}
		\node (0) at (0,0) {$\calR^+(M)$};
		\node (1) at (5.25,0) {$\calR^+(M_\varphi)$};
		\node (2) at (11,0) {$\calR^+((M_\varphi)_{\varphi'})$};
	
		\node (3) at (0,1.2) {$\calR^+(M,D;g_{\tor})$};
		\node (4) at (3.5,1.2) {$\calR^+(M_\varphi, D_\varphi;g_\varphi)$};
		\node (5) at (7,1.2) {$\calR^+(M_\varphi, D_\varphi;\tilde g)$};
		\node (6) at (11,1.2) {$\calR^+((M_\varphi)_{\varphi'}, (D_\varphi)_{\varphi'}; \tilde g_{\varphi'})$};

		\node (7) at (11,-1.2) {$\calR^+(M)$};
		
		\draw[->, dashed] (0) to node[above] {$\overline\calS_\varphi$} (1);
		\draw[-, color=white, line width=20](1) to (2);
		\draw[->,dashed] (1) to node[above] {$\overline\calS_\varphi'$}(2);
		\draw[right hook->] (3) to node[left] {$\iota$} (0);
		\draw[->] (3) to (4);
		\draw[right hook->] (4) to (1);
		\draw[->] (4) to node[above]{(1)} (5);
		\draw[left hook->] (5) to (1);
		\draw[->] (5) to (6);
		\draw[right hook->] (6) to (2);

		\draw[->] (2) to node[right]{$f^*$} (7);
	\end{tikzpicture}
	\end{center}
	Here, the map $(1)$ is given by cutting out $g_\varphi$ and gluing in $\tilde g$. Since these are homotopic relative to the boundary, the inside triangle and hence the entire diagram commutes up to homotopy. Therefore, the composition $f^*\circ\overline\calS_{\varphi'}\circ\overline\calS_\varphi\circ \iota$ is homotopic to the inclusion if and only if the top row composition in this diagram is. In contrast to $f^*\circ\overline\calS_{\varphi'}\circ\overline\calS_\varphi\circ \iota$ this composition only consists of actual maps which are given as follows: For $h\in\calR^+(M\setminus D)_{g_\circ}$ we have
	\begin{center}
	\begin{tikzpicture}
		\node (3) at (0,1.2) {$h\cup g_{\tor}$};
		\node (4) at (3,1.2) {$h\cup g_\varphi$};
		\node (5) at (6,1.2) {$h\cup \tilde g$};
		\node (6) at (9,1.2) {$h\cup \tilde g_{\varphi'}$};
		\node (7) at (9,0) {$h\cup f^*\tilde g_{\varphi'}$};
		
		\draw[|->] (3) to (4);
		\draw[|->] (4) to (5);
		\draw[|->] (5) to (6);
		\draw[|->] (6) to (7);
	\end{tikzpicture}
	\end{center}
	We will denote the path component of a psc-metric $g$ on $M$ by $[g]\in\pi_0(\calR^+(M))$. By the above argument it suffices to show that $[f^*\tilde g_{\varphi'}]=[g_{\tor}]\in\pi_0(\calR^+(D)_{g_\circ})$. This is implied by \pref{Lemma}{lem:surgery-map-round-metric} as follows: We can assume that $D\subset S^{d-1}$ is a hemisphere and we have $f^*\circ\overline\calS_{\varphi'}\circ\overline\calS_{\varphi}([g_{\tor}\cup g_{\tor}])\sim [g_{\tor}\cup f^*\tilde g_{\varphi'}]$ by the above argument for $M=S^{d-1}$ and $h=g_{\tor}$. After possibly changing the coordinates of the disk $D$ we may assume the following: If $a^k\colon S^{d-1}\congarrow (S^{k-1}\times D^{d-k})\cup (D^k\times S^{d-k-1})$ is the solid torus decomposition then $a^{k}\circ\varphi$ is given by the inclusion of the first factor and $a^k_\varphi\circ\varphi'\colon S^{k}\times D^{d-k-1}\embeds (S^{k}\times D^{d-k-1})\cup(S^{k}\times D^{d-k-1})$ is also given by the inclusion of the first factor (cf. \pref{Section}{sec:handle-decomposition}). In this case we have $f=\eta_k$. The metric $[g_{\tor}\cup g_{\tor}]$ is homotopic to the round metric by \cite[Lemma 1.9]{walsh_parametrized1} and we have
	\begin{align*}
		[g_{\tor}\cup f^*\tilde g_{\varphi'}]&\sim \eta_k^*\circ\overline\calS_{\varphi'}\circ\overline\calS_{\varphi} ([g_{\tor}\cup g_{\tor}])\\&\sim \eta_k^*\circ\overline\calS_{\varphi'}\circ\overline\calS_{\varphi} ([g_\circ]) \overset{\text{\pref{Lemma}{lem:surgery-map-round-metric}}}\sim [g_\circ]
		 \sim [g_{\tor}\cup g_{\tor}].
	\end{align*}
	Also $g_1\coloneqq g_{\tor}\cup f^*\tilde g_{\varphi'}$ and $g_2\coloneqq g_{\tor}\cup g_{\tor}$ are both in the image of the inclusion map $\calR^+(D)_{g_\circ}\embeds\calR^+(S^{d-1})$ which is a weak equivalence and since $[g_1]=[g_2]$ it follows that $[g_{\tor}]=[f^*\tilde g_{\varphi'}]\in\pi_0(\calR^+(D)_{g_{\circ}})$.
\end{proof}	

	\begin{lem}\label{lem:surgery-map-round-metric}
		Let $g_\circ\in\calR^+(S^{d-1})$ be the round metric and let $a^k\colon S^{d-1}\congarrow (S^{k-1}\times D^{d-k})\cup (D^{k}\times S^{d-k-1})$ be the solid torus decomposition. Let $\varphi\colon S^{k-1}\times D^{d-k}\embeds S^{d-1}$ and let $\varphi'\colon S^{k}\times D^{d-k-1}\embeds S^{d-1}_\varphi$ be surgery data such that $a^k\circ\varphi$ and $a^k_\varphi\circ\varphi'$ are both given by the inclusion of the respective first factor. Then $\overline\calS_{\varphi'}\circ\overline\calS_{\varphi} ([g_\circ]) \sim \overline\calS_{\eta_k}([g_\circ]) = (\eta_k)_*[g_\circ]$.
	\end{lem}
	
\begin{proof}
	Let $g^k_{\mathrm{mtor}}\coloneqq(g_\circ^{k-1} + g_{\tor}^{d-k})\cup (g_{\tor}^k + g_\circ^{d-k-1})$ denote the \emph{mixed torpedo metric} on$(S^{k-1}\times D^{d-k})\cup (D^{k}\times S^{d-k-1})$. By \cite[Lemma 1.9]{walsh_parametrized1}) we have $(a^k){}^*g_{\mtor}^k\sim g_\circ$ and hence
	\begin{align*}
		\overline\calS_\varphi(g_\circ) &\sim \overline\calS_\varphi\bigl((a^k){}^*g^k_{\mtor}\bigr) = \overline\calS_\varphi(\overline\calS_{(a^k)^{-1}}(g^k_{\mtor}))\\
			&\sim \overline\calS_{(a_\varphi^k{})^{-1}}\overline\calS_{a^k{}\circ\varphi}(g^k_{\mtor}) = (a_\varphi^k)^*\overline\calS_{a^k{}\circ\varphi}(g^k_{\mtor})
	\end{align*} 
	Now $a^k\circ\varphi$ is given by the inclusion and hence 
	\[\overline\calS_{a^k\circ\varphi}(g^k_{\mtor})\sim (g_{\tor}+g_\circ)\cup (g_{\tor} + g_\circ) \sim g_\circ + g_\circ \sim \underbrace{(g_{\circ}+g_{\tor})\cup (g_{\circ} + g_{\tor})}_{\eqqcolon\overline g}\]
	on $(D^k\times S^{d-k-1})\cup(D^k\times S^{d-k-1}) = S^k\times S^{d-k-1} = (S^k\times D^{d-k-1})\cup (S^k\times D^{d-k-1})$. We can now compute
	\begin{align*}
		\overline\calS_{\varphi'}\overline\calS_\varphi(g_\circ) &\sim (a_\varphi^k{})_{\varphi'}{}^*\ \underbrace{\overline\calS_{(a^k_\varphi)\circ\varphi'}\Bigl((g_{\circ}+g_{\tor})\cup (g_{\circ} + g_{\tor})\Bigr)}_{=(g_{\tor}+g_{\circ})\cup (g_{\circ} + g_{\tor}) = g^{k+1}_{\mtor}}\\
		&\sim (a^k_\varphi)_{\varphi'}{}^*\ g^{k+1}_{\mtor}
	\end{align*}
	We have to show that $(a^k_\varphi)_{\varphi'}{}^*\ g^{k+1}_{\mtor}\sim \eta_k{}_*g_\circ$ which is equivalent to 
	\[\eta_k^*\left((a^k_\varphi)_{\varphi'}{}\right)^*\ g^{k+1}_{\mtor}\sim g_\circ\]
	But $\eta_k$ was chosen such that $\bigl((a^k_\varphi)_{\varphi'}\circ\eta_k\bigr)=a^{k+1}$ and therefore 
	\begin{align*}
		\eta_k^*(a^k_\varphi)_{\varphi'}{}^*g^{k+1}_{\mtor}&=(a^{k+1})^*g^{k+1}_{\mtor} \sim g_\circ. \qedhere
	\end{align*}
\end{proof}
	
\noindent We get the following Corollary which follows immediately from \pref{Lemma}{lem:S-is-well-defined} and \pref{Theorem}{thm:presentation-of-bordism-category-ab}. 

\begin{cor}\label{cor:S-is-well-defined}
	Let $d\ge7$. Then there is a unique functor\footnote{By abuse of notation, we call this functor $\overline\calS$ again.}
	\[\overline\calS\colon\bord_d^{-1,2}\longrightarrow\hotop\]
	which satisfies:
	\begin{enumerate}
		\item $\overline\calS(M) = \calR^+(M)$
		\item $\overline\calS_{(M\times I, \id,f)} =  f_*$
		\item $\overline\calS_{(\tr\varphi,\id,\id)}(g) = \overline\calS_\varphi$.
	\end{enumerate}
\end{cor}

\begin{cor}\label{cor:S-exists}
	Let $W = (W,\psi_0,\psi_1)\colon M_0\leadsto M_1$ be an admissible cobordism. Then there is a well defined homotopy class of a map $\overline\calS_W\colon \calR^+(M_0)\to\calR^+(M_1)$. If $W^{\op}\coloneqq (W^{\op},\psi_1,\psi_0)$ is also admissible, i.e. $\psi_0^{-1}\colon M_0\embeds W$ is also $2$-connected, then $\overline\calS_{W}$ is a homotopy equivalence and a homotopy-inverse is given by $\calS_{W^\op}$. 
\end{cor}

\begin{remark}\label{rem:walsh}
	The constructions from the proof of \cite[Theorem 3.1]{walsh_parametrized1} show the following: If $W = (W,\id,\id)\colon M_0\leadsto M_1$ is an admissible cobordism, $g_0\in\calR^+(M_0)$ and $g_1\in\calR^+(M_1)$ are metrics such that $\overline\calS_W([g_0])\sim [g_1]$, then there exists a metric $G\in\calR^+(W)_{g_0, g_1}$.
\end{remark}

\subsection{Surgery invariance of \texorpdfstring{$\overline\calS$}{the surgery map}}\label{sec:surgeryinvariance}

In this section we prove the following Lemma.

\begin{lem}\label{lem:surgeryinvariance}
	Let $d\ge7$ and let $M_0$, $M_1$ be two $(d-1)$-manifolds, let $W=[W,\id,\id]\in\mor{\bord_d^{-1,2}}{M_0,M_1}$ and let $\Phi\colon S^{k-1}\times D^{d-k+1}\embeds \mathrm{Int }\ W$ be an embedding with $3\le k\le d-3$. Then $\overline\calS_W \sim \overline\calS_{W_\Phi}$.
\end{lem}

\begin{proof}
	First we note that for $3\le k\le d-3$, $W_\Phi$ is again an admissible cobordism: Let $W^\circ\coloneqq W\setminus \im\Phi$. Then $W^\circ\embeds W$ is $(d-k)$-connected and $W^\circ\embeds W_\Phi$ is $(k-1)$-connected. We have the following diagram:
	\begin{center}
	\begin{tikzpicture}
		\node (0) at (0,1.2) {$W$};
		\node (1) at (4,1.2) {$W^\circ$};
		\node (2) at (8,1.2) {$W_\Phi$};
		\node (3) at (4,0) {$M_1$};
		
		\draw[->] (3) to node[below, sloped] {$2$-connected} (0);
		\draw[->] (3) to node[auto] {} (1);
		\draw[->] (3) to node[auto] {} (2);
		\draw[->] (1) to node[above] {$(d-k)$-connected} (0);
		\draw[->] (1) to node[above] {$(k-1)$-connected} (2);
	\end{tikzpicture}
	\end{center}	
	Since $3\le k \le d- 3$, the inclusions $M_1\embeds W^\circ$ and $M_1\embeds W_\Phi$ are $2$-connected and hence $W_\Phi$ is admissible.
	
	We first prove \pref{Lemma}{lem:surgeryinvariance} in the case that $k\ne3$. Let $c\colon M_1\times[1-\epsilon,1]\embeds W$ be a collar which does not intersect $\im\Phi$ and let $\gamma\colon[0,1]\times D^{d-1}\embeds W$ be an embedded, thickened path connecting $M_1\times\{1-\epsilon\}$ and $\im\Phi$. Let 
	\begin{align*}
		W_1&\coloneqq\im c\ \#_\partial\ \im\Phi \coloneqq \im c\cup\im\gamma\cup\im\Phi\\
		W_1'&\coloneqq\im c\ \#_\partial\ \im\Phi^{\op}\\
		W_0&\coloneqq W\setminus W_1.
	\end{align*}
	We choose $\gamma$, so that the boundaries of all of these are smooth. Then $W_1\simeq M_1\vee S^{k-1}$, $W_1' \simeq M_1\vee S^{d-k}$, $W_0\cup W_1 = W$ and $W_0\cup W_1' = W_\Phi$. Since $M_1\embeds W$ and $M_1\embeds W_\Phi$ are $2$-connected and $4\le k\le d-3$, the maps $M_1\vee S^{k-1}\simeq W_1\embeds W$ and $M_1\vee S^{d-k}\simeq W'\embeds W_\Phi$ are $2$-connected as well.
		\begin{figure}[ht]
		\centering
		\includegraphics[width=1\textwidth]{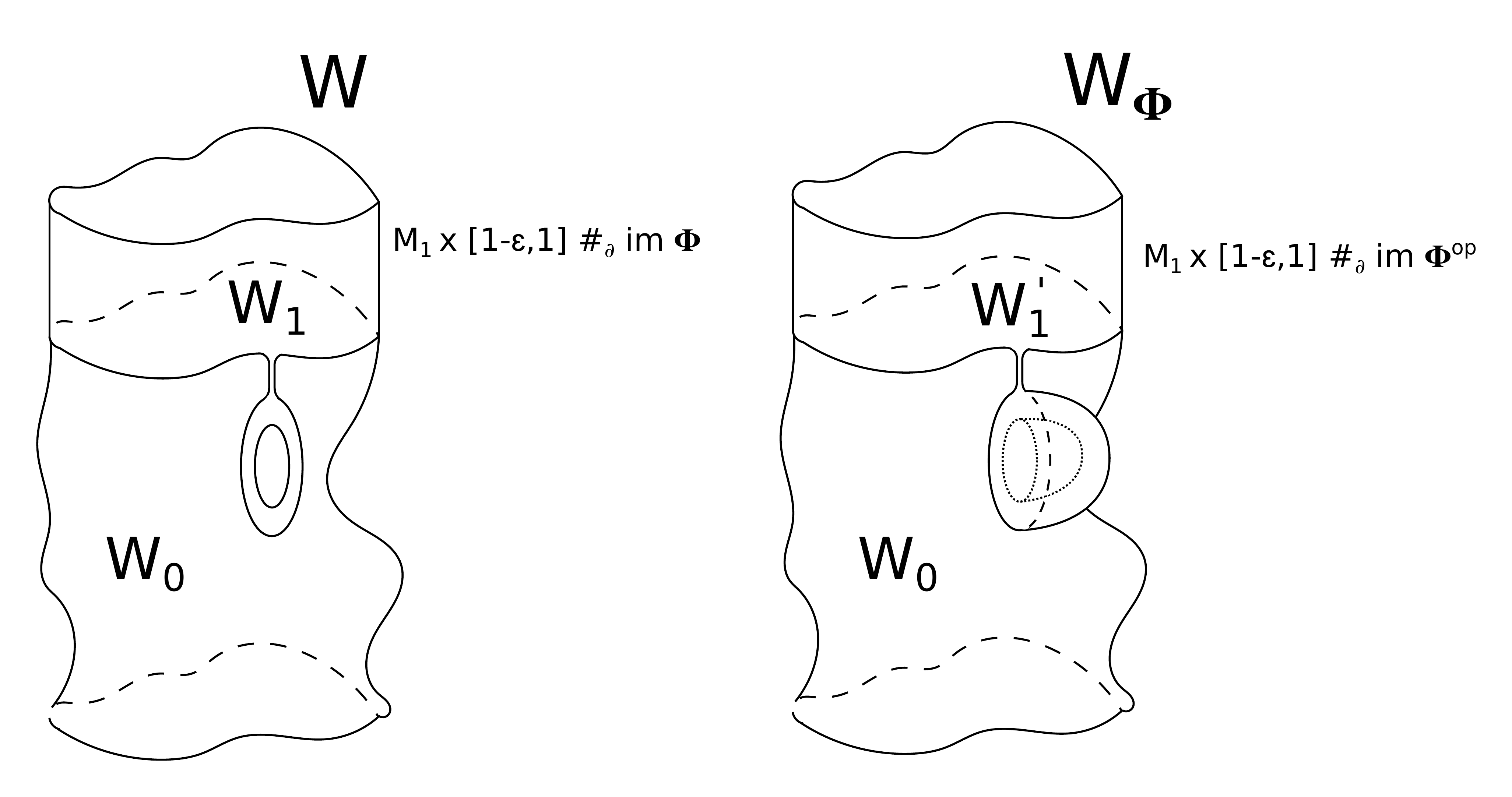}
		\caption{Surgery on the cobordism $W$}
		\end{figure}\\
	Note that $W_1$ and $W_1'$ have the same boundary $M_1'$ given by
	\[\partial W_1 = M_1\amalg \underbrace{(M_1\ \#\ \partial(\im\Phi))}_{\eqqcolon M_1'\cong M_1\# (S^{k-1}\times S^{d-k})} = M_1\amalg (M_1\ \#\ \partial(\im\Phi^{\op} )) = \partial W_1'.\]
	Next, we show that $W_0$, $W_1$, $W_1'$ and $W_1^{\op}$ are again admissible. Because of $W_1\simeq M_1\vee S^{k-1}$ and $W_1'\simeq M_1\vee S^{d-k}$ we get
	\begin{itemize}
		\item[-] $(W_1,M_1)$ is $(k-2)$-connected.
		\item[-] $(W_1,M_1')$ is $(d-k)$-connected.
		\item[-] $(W_1',M_1)$ is $(d-k-1)$-connected.
	\end{itemize}
	So, for $4\le k\le d-3$ all of these are at least $2$-connected and hence $W_1$, $W_1'$ and $W_1^{\op}$ are admissible\footnote{For $k=3$, the cobordism $W_1$ might not be admissible which is why this case is treated separately.}. For $W_0$ we note that $W$ is homotopy equivalent to $W_0$ with a $(d-k+1)$-cell attached along $\Phi(\{1\}\times S^{d-k})$:
	\begin{align*}
		W_0\cup D^{d-k+1} &= \bigl(W\setminus (\im\Phi\cup\im\gamma)\bigr)\cup D^{d-k+1}\\
			&=W\setminus (\underbrace{\im\Phi\setminus D^{d-k+1}}_{\simeq D^{d}}\cup\ \im\gamma)\simeq W.
	\end{align*}
	Therefore $W_0\embeds W$ is $(d-k)$-connected and we have the following diagram.
	\begin{center}
	\begin{tikzpicture}
		\node (0) at (0,0) {$M_1'$};
		\node (1) at (5,0) {$W_1$};
		\node (2) at (0,1.5) {$W_0$};
		\node (3) at (5,1.5) {$W$};
		
		\draw[right hook->] (0) to node[above]{$2$-connected} (1);
		\draw[right hook->] (0) to (2);
		\draw[right hook->] (1) to node[right]{$2$-connected} (3);
		\draw[right hook->] (2) to node[above]{$(d-k)$-connected} (3);
	\end{tikzpicture}
	\end{center}
	and hence $M_1'\embeds W_0$ is $2$-connected, too. 
	
	So we get a decompositions into admissible cobordisms $W=W_0\cup W_1$ and $W_\Phi = W_0\cup W_1'$ which implies $\overline\calS_W = \overline\calS_{W_1} \circ \overline\calS_{W_0}$ and $\overline\calS_{W_\Phi} = \overline\calS_{W_1'} \circ \overline\calS_{W_0}$. In the homotopy category $\hotop$ we have 
	\begin{align*}
		\overline\calS_{W_\Phi} &=  \overline\calS_{W_1'}\circ \underbrace{\overline\calS_{W_1\cup W_1^{\op}}}_{=\id} \circ \overline\calS_{W_0}\\
			&= \overline\calS_{W_1'}\circ\overline\calS_{W_1^{\op}} \circ \overline\calS_{W_1} \circ \overline\calS_{W_0} = \overline\calS_{W_1^{\op}\cup W_1'} \circ \overline\calS_{W},
	\end{align*}
	where $W_1^{\op}\cup W_1'$ denotes the manifold obtained by gluing the outgoing boundary of $W_1^\op$ to the incoming boundary of $W_1'$ along $\id_{M_1'}$. It suffices to show that $W_1^{\op}\cup W_1'$ is diffeomorphic to $M_1\times I$ relative to the boundary since $\overline\calS_W$ only depends on the diffeomorphism type of $W$ (see \pref{Lemma}{lem:S-is-well-defined} and \pref{Corollary}{cor:S-is-well-defined}).
		We have (see \pref{Figure}{fig:surgeryonbordismglued})
		\begin{align*}
			W_1^{\op}\cup W_1'  ={}& \left((M_1\times[0,\epsilon])\ \#_\partial\ S^{k-1}\times D^{d-k+1}\right)\\
				& \underset{M_1'}\cup \left(D^{k}\times S^{d-k}\ \#_\partial\ (M_1\times [1-\epsilon,1])\right)\\
				\cong{}& M_1\times [0,2\epsilon]\ \#\ \underbrace{\left((S^{k-1}\times D^{d-k+1})\underset{S^{k-1}\times S^{d-k}}\cup (D^{k}\times S^{d-k})\right)}_{\cong S^d}\\
				\cong{}& M_1\times [0,1].
		\end{align*}
		and these diffeomorphisms are supported on a small neighbourhood of $M_1'$ and hence relative to the boundary. This finishes the proof for the case $k\ne 3$. 
		\begin{figure}[ht]
		\centering
		\includegraphics[width=25em]{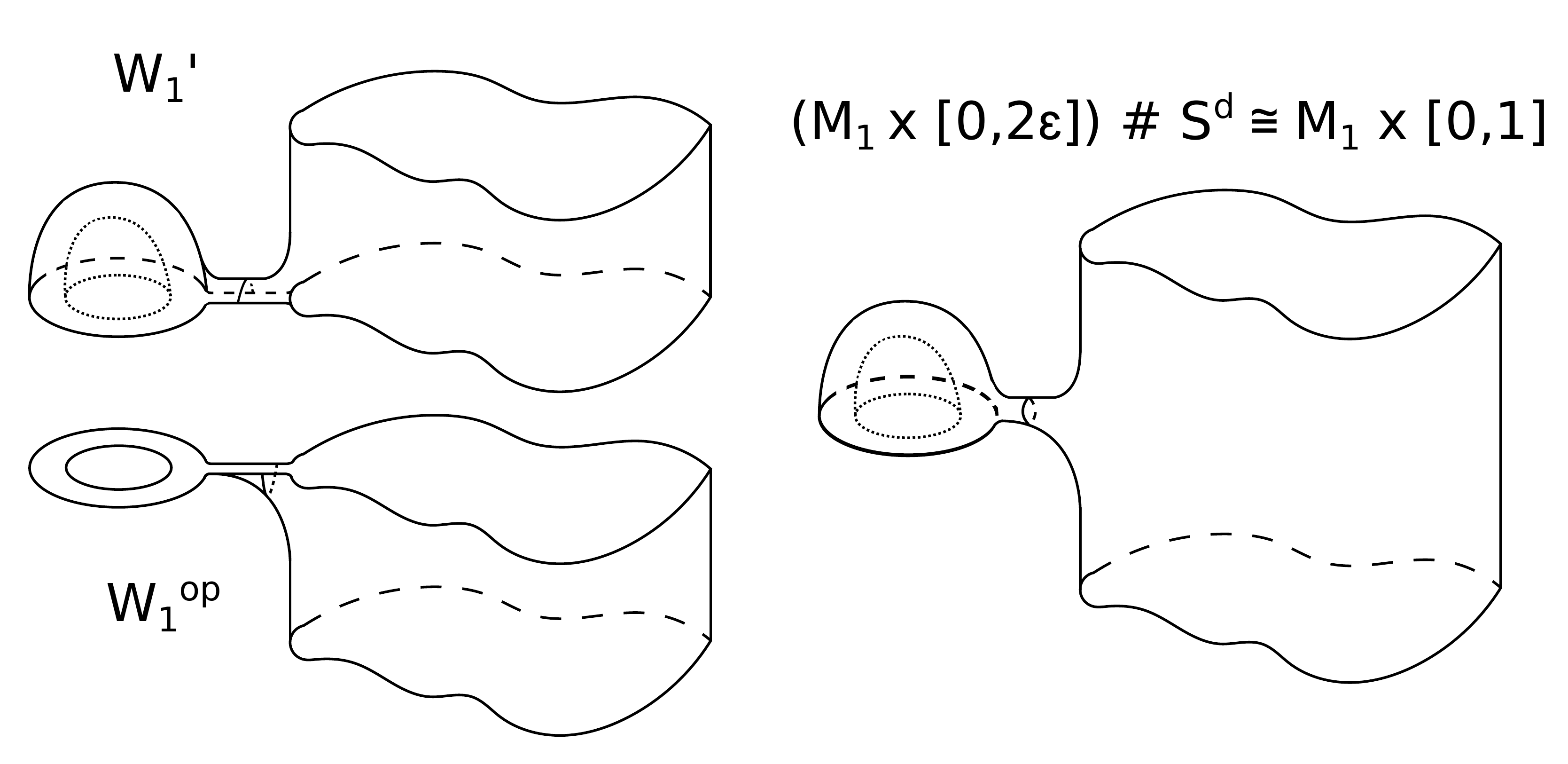}
		\caption{Gluing $W_1^{\op}$ to $W_1'$}\label{fig:surgeryonbordismglued}
		\end{figure}\\
		For the case $k=3$ we need a different argument, because $W_1$ might not be admissible in this case. Consider the map
		\[\emb(S^2\times D^{d-3}, M_1) \too \emb (S^2\times D^{d-2}, M_1\times [0,2])\]
		which is given by $\varphi\mapsto\Phi$ with $\Phi(x,(y,t))=(\varphi(x,y),t)$ for $x\in S^2$ and $(y,t)\in D^{d-2}\subset D^{d-3}\times [0,1]$. We also have a map $\emb (S^2\times D^{d-2}, M_1\times [0,2])\embeds\emb(S^2\times D^{d-2}, W)$ given by shrinking the interval and composing with the inclusion of the collar. We will use the following Lemma.
	\begin{lem}\label{lem:connectivity-of-embeddings}
		In the present situation, the maps $\emb(S^2\times D^{d-3}, M_1) \too \emb (S^2\times D^{d-2}, M_1\times [0,2])$ and $\emb(S^2\times D^{d-2}, M_1\times [0,2])\embeds\emb(S^2\times D^{d-2}, W)$ are both $0$-connected.
	\end{lem}
	\noindent By this Lemma we may isotope the embedding $\Phi\colon S^2\times D^{d-2}\embeds W$ so that its image is contained in the collar of the boundary $M_1$. So we may assume that $W=M_1\times[0,2]$. We abbreviate $M\coloneqq M_1$. Again by the above lemma, we can isotope $\Phi$ such that $\Phi(S^2\times D^{d-3}\times\{0\})\subset M\times \{1\}$, \ie $\Phi$ is a thickening of $\varphi\coloneqq\Phi|_{S^2\times D^{d-3}\times\{0\}}$. Let us now give the diffeomorphism
		{\small\[\underbrace{(M\times[0,\tfrac12]\underset{\varphi}\cup D^3\times D^{d-3})}_{\cong \tr\varphi} \cup \underbrace{(M\times[\tfrac12,1]\underset{\varphi}\cup D^3\times D^{d-3})}_{\cong (\tr\varphi)^{\op}}\overset{\alpha}\too \underbrace{(M\times I)\setminus\im\Phi\cup D^3\times S^{d-3}}_{\cong (M\times I)_\Phi}.\]}
		On $(M\setminus\im\varphi)\times I$ the diffeomorphism $\alpha$ shall be given by the identity. Next we take diffeomorphisms 
		\begin{align*}
			\alpha_1&\colon\im\varphi\times[0,\frac12]\congarrow(\im\varphi\times[0,\frac12])\setminus(\im\Phi\cap[0,\frac12])\\
			\alpha_2&\colon\im\varphi\times[\frac12,1]\congarrow(\im\varphi\times[\frac12,1])\setminus(\im\Phi\cap[\frac12,1]).
		\end{align*}
		On the $D^3\times D^{d-3}$-parts it is given by the inclusion of the lower or upper hemisphere $D^{3}\times S^{d-3}_\pm\subset D^{3}\times S^{d-3}$. The entire diffeomorphism is visualized in \pref{Figure}{fig:diffeomorphism-special-case}. Therefore we have $\overline\calS_{(M\times I)_\Phi}\sim \overline\calS_{\tr\varphi^{\op}}\circ \overline\calS_{\tr\varphi}\sim \id\sim \overline\calS_{M\times I}$ and the proof is finished modulo \pref{Lemma}{lem:connectivity-of-embeddings}.\qedhere
	\begin{figure}[ht]
		\centering
		\includegraphics[width=0.6\textwidth]{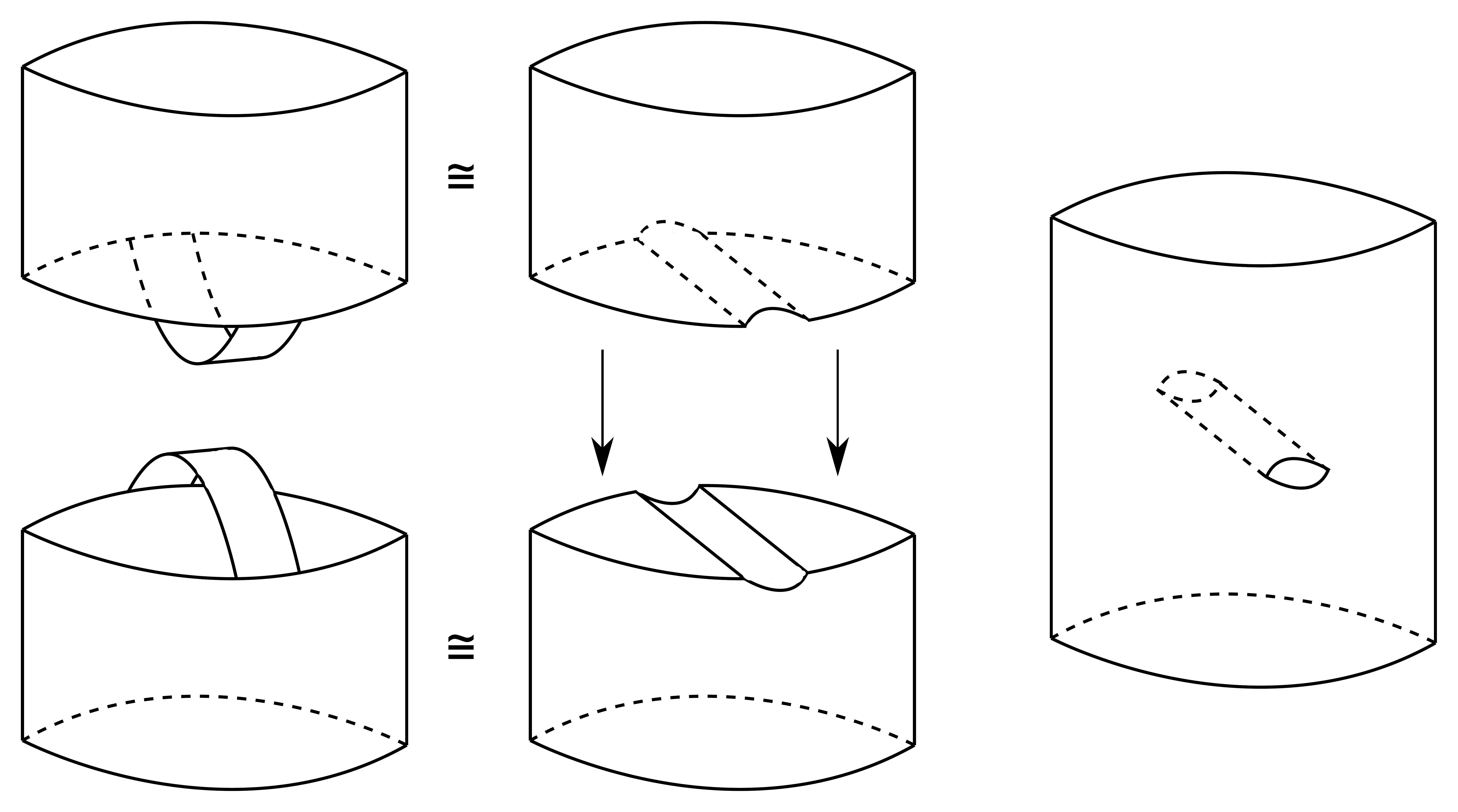}
		\caption{The diffeomorphism $\alpha$.}\label{fig:diffeomorphism-special-case}
		\end{figure}
	\end{proof}
	\begin{samepage}
	\begin{proof}[Proof of \pref{Lemma}{lem:connectivity-of-embeddings}]
	We have the following diagram
		\begin{center}
		\begin{tikzpicture}
			\node (01) at (0,3.3) {$\emb(S^2\times D^{d-3}, M_1)$};
			\node (11) at (5,3.3) {$\emb(S^2\times D^{d-2}, M_1\times [0,2])$};
			\node (0) at (0,2.2) {$\imm(S^2\times D^{d-3}, M_1)$};
			\node (2) at (0,1.1) {$\mon(TS^2\oplus \underline{\bbR}^{d-3}, TM_1)$};
			\node (4) at (0,0) {$\map(S^2, Fr(TM_1) )$};
			\node (5) at (5,0) {$\map(S^2, Fr(TM_1\oplus\bbR))$};
			\node (21) at (10,3.3) {$\emb(S^2\times D^{d-2}, W)$};
			\node (1) at (10,2.2) {$\imm(S^2\times D^{d-2}, W)$};
			\node (3) at (10,1.1) {$\mon(TS^2\oplus \underline{\bbR}^{d-2}, TW)$};
			\node (15) at (10,0) {$\map(S^2, Fr(W))$};

			\draw[->] (01) to node[below]{(4)} (11);
			\draw[->] (01) to node[auto]{(1)} (0);
			\draw[->] (21) to node[auto]{(6)} (1);
			\draw[->] (11) to node[below]{(5)} (21);
			\draw[->] (5) to node[auto]{(3)} (15);
			\draw[->] (0) to node[auto]{$\simeq$} (2);
			\draw[->] (1) to node[auto]{$\simeq$} (3);
			\draw[->] (2) to node[auto]{$\cong$} (4);
			\draw[->] (3) to node[auto]{$\cong$} (15);
			\draw[->] (4) to node[auto]{(2)} (5);
		\end{tikzpicture}
		\end{center}
		where $\mon$ denotes the space of bundle monomorphisms. Note that the bottom-most vertical maps are homeomorphisms because $S^2$ is stably parallelizable and the middle ones are homotopy equivalences by the Smale-Hirsch immersion theorem (cf. \cite[Section 3.9]{adachi_embeddings}). The  map (1) is $0$-connected because of the Whitney embedding (cf. \cite[pp. 26]{hirsch}) and the maps (5) and (6) are $\pi_0$-bijections again by the Whitney-embedding theorem. It remains to show that (2) and (3) are $0$-connected. Then the map (4) is $0$-connected, too. For (2) consider the following diagram of fibrations.
		\begin{center}
		\begin{tikzpicture}
			\node (2) at (0,2) {$\map(S^2, \mathrm{Gl}_{d-1}(\bbR))$};
			\node (3) at (7,2) {$\map(S^2, \mathrm{Gl}_d(\bbR))$};
			\node (4) at (0,1) {$\map(S^2, \mathrm{Fr}(TM_1))$};
			\node (5) at (7,1) {$\map(S^2, \mathrm{Fr}(TM_1\oplus \underline\bbR))$};
			\node (6) at (0,0) {$\map(S^2, M)$};
			\node (7) at (7,0) {$\map(S^2, M)$};
			
			\draw[->](2) to node[above]{$d-4$-conn.}(3);
			\draw[->](2) to node[left]{} (4);
			\draw[->](3) to node[left]{} (5);
			\draw[->](4) to node[right]{} (5);
			\draw[->](4) to node[right]{} (6);
			\draw[->](5) to node[left]{} (7);
			\draw[double equal sign distance](6) to (7);
		\end{tikzpicture}
		\end{center}
		Since $d-4\ge 3$, the map (2) is $0$-connected. The map (3) fits into a similar diagram:
		\begin{center}
		\begin{tikzpicture}
			\node (2) at (0,2) {$\map(S^2, \mathrm{Gl}_{d}(\bbR))$};
			\node (3) at (7,2) {$\map(S^2, \mathrm{Gl}_d(\bbR))$};
			\node (4) at (0,1) {$\map(S^2, \mathrm{Fr}(TM_1\oplus\underline\bbR))$};
			\node (5) at (7,1) {$\map(S^2, \mathrm{Fr}(W))$};
			\node (6) at (0,0) {$\map(S^2, M_1)$};
			\node (7) at (7,0) {$\map(S^2, W)$};
			
			\draw[double equal sign distance](2) to (3);
			\draw[->](2) to node[left]{} (4);
			\draw[->](3) to node[left]{} (5);
			\draw[->](4) to node[right]{} (5);
			\draw[->](4) to node[right]{} (6);
			\draw[->](5) to node[left]{} (7);
			\draw[](6) to (7);
		\end{tikzpicture}
		\end{center}
		Since $M_1\embeds W$ is $2$-connected, the bottom-most map is $0$-connected and hence so is the map (3). 
	\end{proof}
	\end{samepage}

\section{Tangential structures and proof of main result}\label{sec:surgerymap}

\subsection{Tangential structures}\label{sec:tangential-structures}

	In order to get rid of the connectivity assumptions of the category $\bord_d^{-1,2}$, we need tangential structures. For $d\ge0$ let $B\ort(d+1)$ be the classifying space of the $(d+1)$-dimensional orthogonal group and let $U_{d+1}$ be the universal vector bundle over $B\ort(d+1)$. Let $\theta\colon B\to B\ort(d+1)$ be a fibration. We call $\theta$ a \emph{tangential structure}.
	
	\begin{definition}
			A \emph{$\theta$-structure} on a real $\rk {d+1}$-vector bundle $V\to X$ is a bundle map $\hat l\colon V\to \theta^*U_{d+1}$. A \emph{$\theta$-structure on a manifold} $W^{{d+1}}$ is a $\theta$-structure on $TW$ and a \emph{$\theta$-manifold} is a pair $(W,\hat l)$ consisting of a manifold $W$ and a $\theta$-structure $\hat l$ on $W$. For $0\le k\le d$ a \emph{stabilized $\theta$-structure} on $M^k$ is a $\theta$-structure on $TM\oplus\underline\bbR^{{d+1}-k}$. 
	\end{definition}
	
	\noindent An important source of tangential structures are covers of $B\ort({d+1})$. For example we have $B\SO({d+1})\to B\ort({d+1})$ or $B\Spin({d+1})\to B\ort({d+1})$ or more generally $B\ort({d+1})\langle k\rangle\to B\ort({d+1})$, where $B\ort({d+1})\langle k\rangle$ denotes the $k$-connected cover of $B\ort({d+1})$. Other sources of tangential structures are Moore-Postnikov towers:
	
	\begin{definition}
		Let $M^{d-1}$ be a connected manifold, let $l\colon M\to B\ort({d+1})$ be the classifying map of the stabilized tangent bundle and let $\hat l\colon TM\oplus\underline\bbR^2\to U_{d+1}$ be a bundle map covering $l$. The $n$-th stage of the Moore-Postnikov tower for the map $l$ is called the \emph{stabilized tangential $n$-type of $M$}. 
	\end{definition} 

	\begin{example}\label{ex:tangential-structures}\leavevmode
		\begin{enumerate}
			\item The stabilized tangential $2$-type of a connected $\Spin$-manifold $M$ of dimension at least $3$ is $B\Spin({d+1})\times B\pi_1(M)$.
			\item The stabilized tangential $2$-type of a simply connected, non-spinnable manifold $M$ of dimension at least $3$ is $B\SO({d+1})$.
		\end{enumerate}
	\end{example}
	
	\noindent Recall the following lemma which is frequently used when working with surgery results concerning positive scalar curvature.

\begin{lem}[{\cite[Proposition 4]{kreck_duality}, \cite[Proposition, Appendix B]{hebestreitjoachim}, \cite[Lemma B.4]{ownthesis}}]\label{lem:connectivity-of-bordisms}
	Let $\theta\colon B\to B\ort(d+1)$ be a tangential structure, with $B$ of type $F_n$. Let $W^m\colon M_0\leadsto M_1$ be a $\theta$-cobordism and let $M_1\to B$ be $n$-connected. If $n\le\frac m2-1$, there exists a $\theta$-cobordism $W'\colon M_0\leadsto M_1$ such that $(W',M_1)$ is $n$-connected. If furthermore $M_0\to B$ is also $n$-connected, there exists a $\theta$-cobordism $W'\colon M_0\leadsto M_1$ such that $(W',M_i)$ is $n $-connected for $i=0,1$. Furthermore $W'$ is $\theta$-cobordant to $W$ relative to the boundary.
\end{lem}
	
\subsection{Proof of the main result}

We will now prove the general version of \pref{Theorem}{thm:a-general} which is the main result of this article.

\begin{definition}\label{def:omega}
	We define $\Omega_{d,2}$ to be the category given by the following:
	\begin{itemize}
		\item[] Objects are given by tupels $(M,B,\theta,\hat l)$ where 
		\begin{itemize}
			\item[-]$M$ is a closed $(d-1)$-dimensional manifold.
			\item[-]$\theta\colon B\to B\ort(d+1)$ is a $2$-coconnected tangential structure.
			\item[-]$\hat l$ is a stabilized $\theta$-structure such that the underlying map $l\colon M\to B$ is $2$-connected.
		\end{itemize}
		\item[] Morphisms $(M_0,B_0,\theta_0,\hat l_0)$ to $(M_1,B_1,\theta_1,\hat l_1)$ are given by equivalence classes of tupels $(W,\psi_0,\psi_1,\hat \ell, h)$ where
		\begin{itemize}
			\item[-] $h\colon B_0\to B_1$ is a map over $B\ort(d+1)$. This gives an induced map 
				\[\hat h\colon\theta_0^*U_{d+1}\to\theta_1^*U_{d+1}.\]
			\item[-] $(W,\psi_0,\psi_1)$ is a cobordism from $M_0$ to $M_1$.
			\item[-] $\hat\ell$ is a stabilized $\theta_1$-structure on $W$.
			\item[-] $\hat\ell|_{\partial_0W}= \hat h\circ\hat l_0\circ \mathrm{d}\psi_0$ und $\hat \ell|_{\partial_1W}=-\hat l_1\circ \mathrm{d}\psi_1$, where $-\hat l_1$ denotes the bundle map given by 
			\[\hat l_1\circ (\id\oplus\left(\begin{matrix} -1 & \\ & 1 \end{matrix}\right))\colon TM_1\oplus\underline\bbR^2 \to TM_1\oplus\underline\bbR^2\to \theta_1^*U_{d+1}\]
			\begin{figure}[ht]
		\centering
		\includegraphics[width=0.4\textwidth]{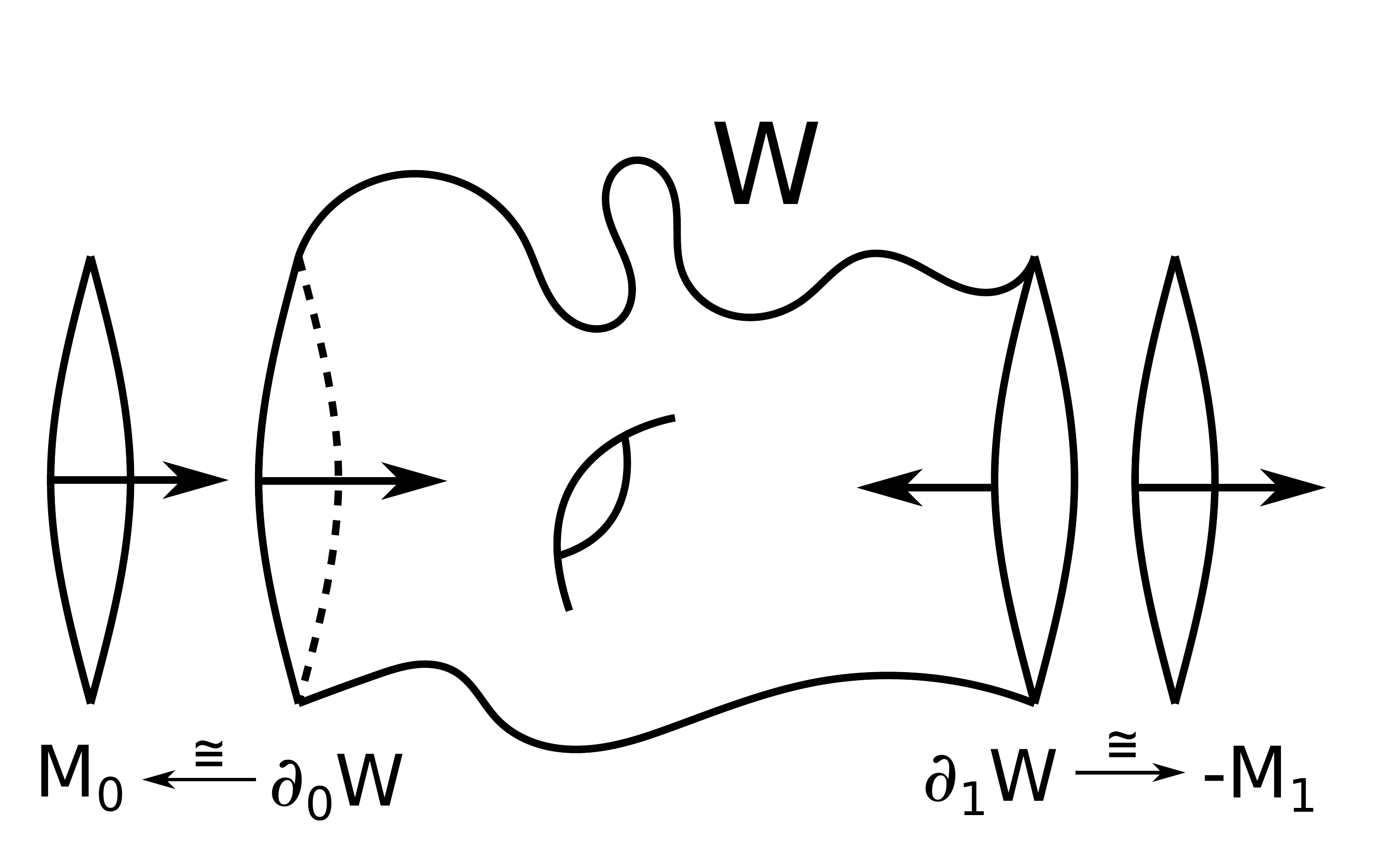}
		\caption{A representative of a morphism in $\Omega_{d,2}$}\label{fig:bordism-set}
		\end{figure}
			\item[-] $(W,\psi_0,\psi_1,\hat \ell, h)\sim (W',\psi_0',\psi_1',\hat \ell', h')$ if $h=h'$ and there exists a $(d+1)$-dimensional $\theta_1$-manifold $(X,\ell_X)$ with corners such that there exists a partition of $\partial X=\bigcup_{i=0,3} \partial_iX$ together with diffeomorphisms
		\begin{align*}
			\partial_0X&\congarrow M_0\times I &\partial_2X &\congarrow M_1\times I\\
			\partial_1X&\congarrow W &\partial_3X &\congarrow W'
		\end{align*}  
		such that $\theta$-structures and diffeomorphisms fit together (see \pref{Figure}{fig:bordism-set-relation}).
		\begin{figure}[ht]
		\centering
		\includegraphics[width=0.7\textwidth]{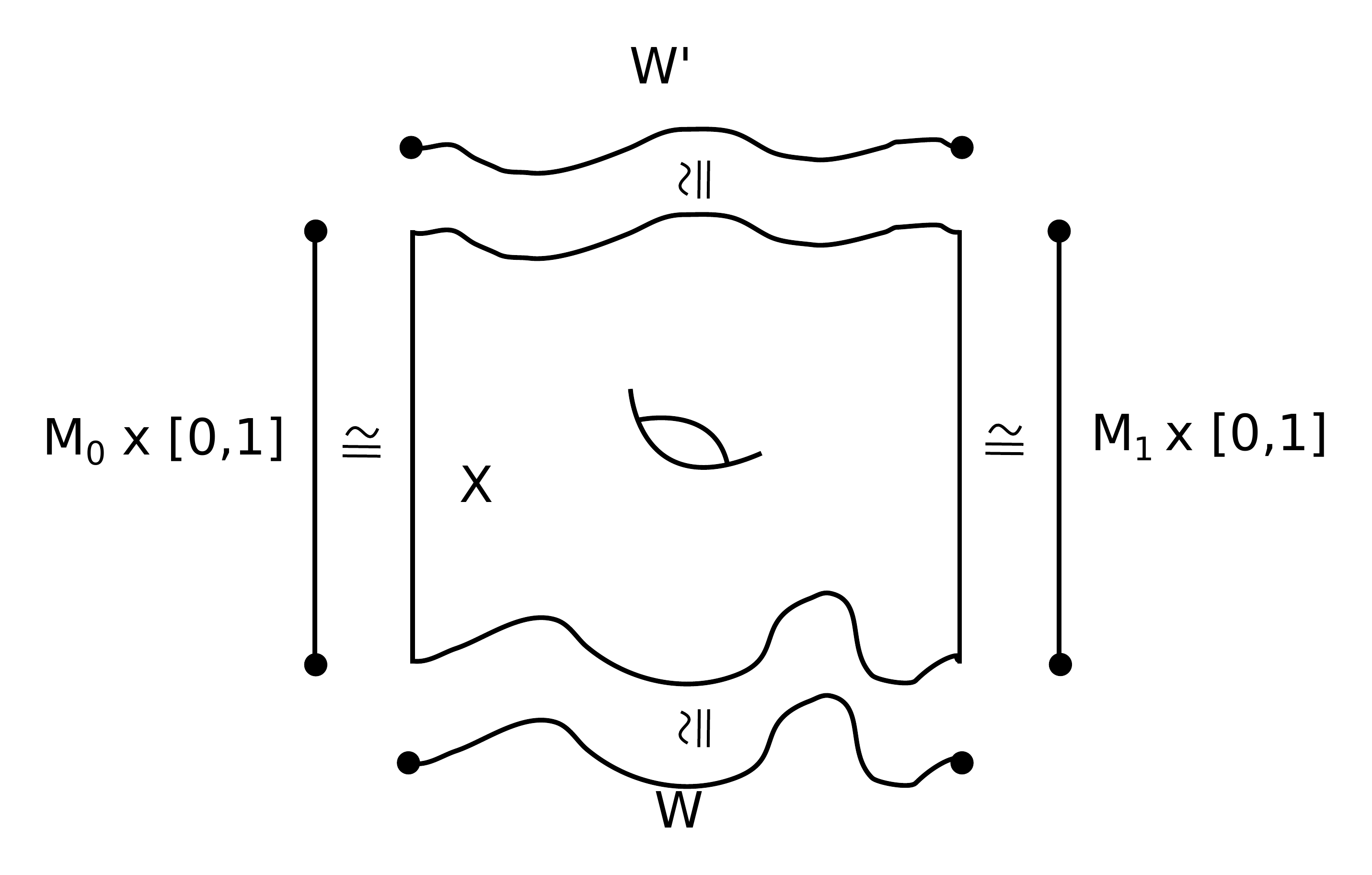}
		\caption{The relation in the category $\Omega_{d,2}$.}\label{fig:bordism-set-relation}
		\end{figure}
		\end{itemize}
		\item[]  Composition is given by gluing cobordisms along the common boundary:
			\[(W',\psi_0',\psi_1',\hat \ell', h')\circ (W,\psi_0,\psi_1,\hat \ell, h) = (W\cup_{(\psi_0')^{-1}\circ\psi_1} W', \psi_0,\psi_1',\hat\ell\cup_{(\psi_0')^{-1}\circ\psi_1}\hat\ell',h'\circ h).\]

	\end{itemize}
\end{definition} 

\begin{thm}\label{thm:main}
	Let $d\ge7$. There is a functor 
	\[\calS\colon \Omega_{d,2}\too \hotop\]
	with the following properties:
	\begin{enumerate}
		\item On objects, $\calS$ is given by $\calS(M,B,\theta,\hat l)=\calR^+(M)$.
		\item If $\alpha\in\Omega_{d,2}\bigl((M_0,B_0,\theta_0,\hat l_0),(M_1,B_1,\theta_1,\hat l_1)\bigr)$ is represented by a cobordism whose underlying manifold is given by $(M_0\times[0,1], \id, f^{-1})$ for a diffeomorphism $f\colon M_1\to M_0$, then $\calS(\alpha) = f^*$.
		\item If $\alpha\in\Omega_{d,2}\bigl((M_0,B_0,\theta_0,\hat l_0),(M_1,B_1,\theta_1,\hat l_1)\bigr)$ is represented by a cobordism whose underlying manifold is given by the trace $(\tr(\varphi),\id,\id)$ of a surgery datum $\varphi\colon S^{k-1}\times D^{d-k}\embeds M_0$ with $d-k\ge3$, then $\calS(\alpha) = \overline\calS_\varphi$ (cf. \pref{Definition}{def:surgeryequivalence}).
	\end{enumerate}
	Furthermore, $\calS$ is uniquely determined by these properties, up to natural isomorphism.
\end{thm}

\begin{proof}
	Let $V\coloneqq(V,\psi_0,\psi_1,\ell_V)\colon (M_0, \hat h\circ \hat l_0)\leadsto (M_1,\hat l_1)$ be a $\theta_1$-cobordism. By \pref{Lemma}{lem:connectivity-of-bordisms}, there exists an admissible $\theta_1$-cobordism $V'\colon M_0\leadsto M_1$ in the same cobordism class. We define $\calS_V\coloneqq\overline\calS_{V'}$. By definition of $\overline\calS$ it is clear that this fulfils the desired properties and is compatible with composition. It remains to show that this is well-defined. Let $X\colon V_0\leadsto V_1$ be a $\theta_1$-cobordism relative to $\partial V_0 = \partial V_1$ and let $X_i\colon V_i\leadsto V_i'$ be relative $\theta_1$-cobordisms such that $(V_i',M_1)$ is admissible for $i=0,1$. We get a relative $\theta_1$-cobordism $\widetilde X\coloneqq X_0^{\op}\cup X\cup X_1\colon V_0' \leadsto V_0 \leadsto V_1\leadsto V_1'$. Again, by \pref{Lemma}{lem:connectivity-of-bordisms}, we may assume that $(\widetilde X,V_i')$ is $2$-connected. So, $V_1'$ is obtained from $V_0'$ by a sequence of surgeries of index $k\in\{3,\dots,d-2\}$. One can order these surgeries, so that one first performs the $3$-surgeries, the $4$-surgeries next and so on up to the $d-3$-surgeries. By \pref{Lemma}{lem:surgeryinvariance} all of these do not change the homotopy class of $\overline\calS$ and we may assume that $V_1'$ is obtained from $V_0'$ by a finite sequence of $d-2$-surgeries. Reversing these surgeries we deduce that $V_0'$ is obtained from $V_1'$ by a finite sequence of $3$-surgeries and by \pref{Lemma}{lem:surgeryinvariance} the map $\overline\calS_{V_0'}$ is homotopic to $\overline\calS_{V_1'}$. Hence $\calS$ is well-defined. 
\end{proof}

\begin{remark}
Note that if $M_0$ and $M_1$ have the same tangential $2$-type, there exists an admissible cobordism $V'$ in the same cobordism class as $V$ such that $(V')^\op$ is admissible as well. Then $\calS_{(V')^\op}$ is an inverse for $\calS_V$
\end{remark}

\begin{remark}
As mentioned in \pref{Remark}{rem:walsh} (see also \cite{walsh_parametrized1}), Walsh constructed a psc-metric $G$ on an admissible self-cobordism $W\colon M\leadsto M$ extending a given psc-metric $g_0$ on the incoming boundary using the same construction used here. He showed that the homotopy class of $G$ restricted to the outgoing boundary does not depend on the handle presentation \cite[Theorem 1.3]{walsh_parametrized2}. Therefore he obtained a map $f_W\in\aut(\pi_0(\calR^+(M)))$ given by $[g_0]\mapsto [G|_{M\times\{1\}}]$. By separating the cobordism part of the picture (\pref{Section}{sec:handle-decomposition} to \pref{Section}{sec:presentation}) from the scalar curvature part of the picture (\pref{Section}{sec:S-is-well-defined} and \pref{Section}{sec:surgeryinvariance}) we upgraded this to give an actual homotopy class of a map $\calS_W\in\pi_0(\haut(\calR^+(M)))$ inducing Walsh's map on $\pi_0(\calR^+(M))$. The second improvement lies in the cobordism-invariance of $\calS$ which drastically enlarges its kernel and enables us to define $\calS_W$ for \emph{any} $\theta$-cobordism $W$.
\end{remark}

\noindent Before we start deriving the general version of \pref{Theorem}{thm:a}, let us list two interesting facts about the surgery map. The first one is proven by an argument similar to the reduction step in the proof of \pref{Lemma}{lem:S-is-well-defined} and uses the notion of left-/right-stable metrics (cf. \cite{erw_psc2}). 

Let $M_0$ be a manifold and let $M_0^{(2)}$ consist of all $0$-, $1$- and $2$-handles of $M_0$. We write $Q_0\coloneqq M_0\setminus M_0^{(2)}$ and $N\coloneqq\partial(M_0\setminus M_0^{(2)})$. We get a decomposition of $M_0$ into two cobordisms $\emptyset\overset{M_0^{(2)}}\leadsto N\overset{Q_0}\leadsto \emptyset$. A metric $g\in \calR^+(M_0^{(2)})_h$ is called \emph{right-stable} if for every cobordism $V\colon N\leadsto N'$ the map $\mu(g,\_)\colon\calR^+(V)_{h,h'}\to \calR^+(M_0^{(2)}\cup V)_{h'}$ which glues in $g$ is a weak homotopy equivalence. Analogously  a metric $g\in \calR^+(Q_0)_h$ is called \emph{left-stable} if for every cobordism $V\colon N'\leadsto N$ the map $\mu(\_,g)\colon\calR^+(V)_{h',h}\to \calR^+(V\cup Q_0)_{h'}$ which glues in $g$ is a weak homotopy equivalence.

\begin{prop}\label{prop:single-metric}
	Let $M_0$ be such that there exists a metric $g=g_{\rst}\cup g_{\lst}\in\calR^+(M_0)$ which is the union of a right-stable metric $g_{\rst}\in \calR^+(M_0^{(2)})_h$ and a left-stable metric $g_{\lst}\in\calR^+(Q_0)_h$. Let $W = (W,\id,\id), W'=(W',\id,\id)\colon M_0\leadsto M_1$ be an admissible cobordisms with $\calS_W(g) \sim \calS_{W'}(g)$. Then $\calS_W$ is homotopic to $\calS_{W'}$.
\end{prop}

\begin{proof}[Proof of \pref{Proposition}{prop:single-metric}]
	Since $W$ and $W'$ are admissible, they consist of handles glued along surgery data with codimension at least $3$. By transversality we may assume that all handles are attached in the interior of $Q_0$. Hence we can decompose $M_1$ into $M_0^{(2)}\cup Q_1$ and $W$ (resp $W'$) into $M_0^{(2)}\times[0,1]$ and a relative cobordism $V\colon Q_0\leadsto Q_1$ (resp. $V'$). Let $g_{\lst}^V$ and $g_{\lst}^{V'}$ represent the resulting path components of $\calS_V(g_{\lst})$ and $\calS_V'(g_{\lst})$. Since $g_\lst$ is left-stable $\mu(\_,g_\lst)$ is a weak equivalence and $\calS_W=\mu(\_,g_\lst)^{-1}\circ\mu(\_,g_\lst^V)$ and $\calS_{W'}=\mu(\_,g_\lst)^{-1}\circ\mu(\_,g_\lst^{V'})$. By assumption $g_{\rst}\cup g^{V}_{\lst}$ is homotopic to $g_{\rst}\cup g^{V'}_{\lst}$ and because $g_\rst$ is right-stable, $g^{V}_{\lst}$ is homotopic to $g^{V'}_{\lst}$. Therefore $\mu(\_,g^{V}_{\lst})\sim \mu(\_,g^{V'}_{\lst})$ and hence $\calS_W\sim \calS_{W'}$.
\end{proof}

\begin{remark}\label{rem:single-metric}
	This theorem applies for example if $M_0$ is the double $dM_0^{(2)}=M_0^{(2)} \cup (M_0^{(2)})^\op$ of $M_0^{(2)}$ and the metric $g$ is given by $g_{\rst}\cup g_\rst^\op$ which covers the case  $M_0=S^{d-1}$ and $g=g_\circ$.   
\end{remark}

\noindent The second fact states that the surgery map induces a well defined map on concordance classes of psc-metrics which will be used in forthcoming work \cite{own_hspace}. Let us first recall the notion of concordance of psc-metrics.

\begin{definition}
	Let $g_0,g_1\in\calR^+(M)$. We say $g_0$ and $g_1$ are \emph{concordant} if $\calR^+(M\times[0,1])_{g_0,g_1}\ne\emptyset$. This defines an equivalence relation and we denote the \emph{set of concordance classes} of $\calR^+(M)$ by $\tilde\pi_0(\calR^+(M))$.
\end{definition}

\begin{prop}
	$\calS$ induces a well defined map on concordance classes.
\end{prop}

\begin{proof}
	Let $M_0, M_1$ be as in \pref{Theorem}{thm:main}, $g,g'\in\calR^+(M_0)$ be concordant metrics via $G\in\calR^+(M_0\times[0,1])_{g,g'}$ and let $[W]\in\Omega_d^{\theta_1}(M_0,M_1)$. Without loss of generality we may assume that $W$ is admissible. The map $\calS_W$ induces a map on components and since isotopy of psc-metrics implies concordance of psc-metrics, there are unique concordance classes $[\calS_W[g]]$ and $[\calS_W[g']]$ represented by $h$ and $h'$ respectively. It remains to show that $h$ and $h'$ are concordant. By \cite[Theorem 3.1]{walsh_parametrized1} (cf. \pref{Remark}{rem:walsh}) there exist metrics $H\in\calR^+(W)_{g,h}$ and $H'\in\calR^+(W)_{g',h'}$. This gives the psc-metric $H'^{\op}\cup G\cup H\in \calR^+(W^{op}\cup M_0\times[0,1]\cup W)_{h',h}$. By \pref{Proposition}{prop:double-is-nullbordant}, $W^{\op}\cup M_0\times[0,1]\cup W$ is $\theta_1$-cobordant to $M_1\times[0,1]$ relative to the boundary and by the surgery theorem there exists a psc-metric $\tilde H\in\calR^+(M_1\times[0,1])_{h',h}$.
\end{proof}

\begin{remark}\label{rem:concordance}
	Let $W\colon M_0\leadsto M_1$ be an admissible cobordism. A similar argument shows that on concordance classes we have 
	\[[\calS_W(g)] = [h] \iff \exists G\in\calR^+(W)_{g,h}\]
\end{remark}

\subsection{The Structured Mapping Class Group}\label{sec:mapping-class-groups}

	In this section we will give the definitions and present two models for the structured mapping class group of a manifold. For the next two sections let $\theta\colon B\to B\ort(d+1)$ be a fixed tangential structure.
	\begin{definition}
		For a smooth manifold $M^{d-1}$ we denote by $\diff(M)$ the \emph{topological group of diffeomorphisms of $M$} with the (weak) $C^\infty$-topology. If $M$ is oriented we denote \emph{the subgroup of orientation preserving diffeomorphisms of $M$} by $\diff^+(M)$. The \emph{(unoriented) mapping class group} $\Gamma(M)$ is defined to be $\pi_0(\diff(M))$ and the \emph{oriented mapping class group} $\Gamma^+(M)$ is defined as $\pi_0(\diff^+(M))$.
	\end{definition}
	
	\begin{definition}
		Let $M^{d-1}$ be a smooth oriented manifold. We define
		\begin{align*}
			E\diff^\theta(M)&\coloneqq E\diff(M)\times\bun(TM\oplus \underline{\bbR}^2,\theta^*U_{d+1})\\
			B\diff^\theta(M)&\coloneqq E\diff^\theta(M)/\diff(M) = E\diff(M)\underset{\diff(M)}{\times}\bun(TM\oplus \underline{\bbR}^2,\theta^*U_{d+1}),
		\end{align*}
		where we use the model $E\diff(M)\coloneqq\{j\colon M\hookrightarrow \bbR^{\infty-1}\}$ which is the (contractible) space of embeddings and $\bun(\_,\_)$ denotes the space of bundle maps. More concretely, 
		\[B\diff^\theta(M)= \{(N, \hat l)\colon N\subset\bbR^{\infty-1},\ N\cong M\text{ and } \hat l\in \bun(TN\oplus \underline{\bbR}^2,\theta^*U_{d+1})\}.\]
		Given an embedding $j\colon M\embeds \bbR^{\infty-1}$ and a (stabilized) $\theta$-structure $\hat l$ on $M$, we get a base-point $(j(M),\hat l)\in B\diff^\theta(M)$. We also define the \emph{universal $M$-bundle with $\theta$-structure $U_{M,\theta}$} by
		\[U_{M,\theta}\coloneqq E\diff^\theta(M)\underset{\diff(M)}\times M \too B\diff^\theta(M).\]
	\end{definition}
	
\begin{remark}
	For $\theta_{B\SO}\colon B\SO({d+1})\to B\ort({d+1})$ we abbreviate $B\diff^{\theta_{B\SO}}(M)$ by $B\diff^+(M)$. Note that with our convention $E\diff^+(M)$ is not contractible but homotopy equivalent to $\bun(TM\oplus\underline\bbR^2, \theta_{B\SO}^*U_{d+1})$ which has two contractible components provided that $M$ is connected (cf. \cite[p. 6]{ownthesis}). However, if $M$ is connected and admits an orientation reversing diffeomorphism, $B\diff^+(M)$ is still a model for the classifying space of $\diff^+(M)$-principal bundles.
\end{remark}
			
	\begin{definition}[Structured Mapping Class Group]\label{def:structured-mapping-class-group}
		Let $M$ be a smooth submanifold of $\bbR^{\infty-1}$ and let $\hat l$ be a stabilized $\theta$-structure on $M$. The \emph{$\theta$-structured mapping class group} $\Gamma^\theta(M, \hat l)$ is defined by
		\[\Gamma^{\theta}(M, \hat l):=\pi_1(B\diff^\theta(M),(M, \hat l)).\] 
		For $\gamma\colon S^1\to B\diff^\theta(M)$ we define the \emph{structured mapping torus} $M_\gamma:=\gamma^*U_{M,\theta}$. 
	\end{definition}
	
	\begin{remark}
		The mapping torus $M_\gamma$ has a $\theta$-structure on the vertical tangent bundle. Since the tangent bundle of the circle is trivial, this gives a $\theta$-structure on $M_\gamma$.
	\end{remark}

	\noindent Since the case of $B=B\Spin({d+1})$ is of great interest in the present paper we will have a closer look at it. Let us recall the more traditional description of $\Spin$-structures (cf. \cite[Chapter 3]{ebert_characteristic}): \emph{A $\Spin$-structure $\sigma$ on a manifold $M$} is a pair $(P,\alpha)$ consisting of a $\Spin(d+1)$-principal bundle $P$ and an isomorphism $\alpha\colon P\times_{\Spin({d+1})} \bbR^{d+1}\congarrow TM\oplus\underline\bbR^2$. An \emph{isomorphism of $\Spin$-structures} $\sigma_0=(P_0,\alpha_0)$ and $\sigma_1=(P_1,\alpha_1)$ is an isomorphism $\beta\colon P_0\congarrow P_1$ of $\Spin({d+1})$-principal bundles over $\id_M$ such that $\alpha_1\circ(\beta\times_{\Spin({d+1})}\id_{\bbR^{d+1}})=\alpha_0$. If $f\colon M\to M$ is an orientation preserving diffeomorphism and $\sigma=(P,\alpha)$ is a $\Spin$-structure on $M$, we define $f^*\sigma\coloneqq(f^*P, (df)^{-1}\circ f^*\alpha)$.

	Now, let $\sigma_0, \sigma_1$ be two $\Spin$-structures of $M$. A $\Spin$-diffeomorphsim $(M,\sigma_0)\congarrow(M,\sigma_1)$ is a pair $(f,\hat f)$ consisting of an orientation preserving diffeomorphism $f\colon M\congarrow M$ and an isomorphism $\hat f$ of $\Spin$-structures $\sigma_0$ and $f^*\sigma_1$ (cf. \cite[Definition 3.3.3]{ebert_characteristic}). We denote by $\diff^\Spin((M,\sigma_0),(M,\sigma_1))$ the space of $\Spin$-diffeomorphisms $(M,\sigma_0)\congarrow(M,\sigma_1)$. This gives rise to the groupoid $\diff^\Spin(M)$ which has $\Spin$-structures on $M$ as objects and morphisms spaces are given by $\diff^\Spin((M,\sigma_0),(M,\sigma_1))$. For a $\Spin$-structure $\sigma$ on $M$, we define
	\[\diff^\Spin(M,\sigma)\coloneqq\diff^\Spin((M,\sigma)(M,\sigma)).\]	
	
	\begin{prop}\label{prop:spin-mcg}
		Let $M$ be a simply connected $\Spin$-manifold. Then the forgetful homomorphism $\diff^\Spin(M,\sigma) \to \diff^+(M)$ is surjective and its kernel has two elements.
	\end{prop}
	\begin{proof}
		Since $M$ is simply connected, the $\Spin$-structure $\sigma$ of an oriented manifold is unique up to isomorphism. So for every orientation preserving diffeomorphism $f\colon M\congarrow M$, there is an isomorphism $\sigma\congarrow f^*\sigma$, hence the map is surjective. The rest follows from \cite[Lemma 3.3.6]{ebert_characteristic}.
	\end{proof}

	\noindent If $\theta$ is an arbitrary tangential structure we also have a different model for $\Gamma^\theta(M,\hat l)$.
	\begin{definition}
		For a $\theta$-structure $\hat l$ on $M^{d-1}$ we define 
		\[B^\theta(M,\hat l)\coloneqq\left\{(f,L)\colon \begin{array}{l}f\colon M\congarrow M\text{ is a diffeomorphism}\\L\text{ is a homotopy of bundle maps } \hat l\circ df\leadsto \hat l \end{array}\right\}\Big/\sim\]
		where the equivalence relation is given by homotopies of $f$ and $L$. 
	\end{definition}
%
%
%
%
	\begin{prop}[{\cite[Proposition 1.2.11]{ownthesis}}]\label{prop:models-for-mcg}
		There is a group isomorphisms $\Gamma^\theta(M,\hat l)\congarrow B^\theta(M,\hat l)$.
	\end{prop}

%

	\begin{example}
Since we usually will be interested in the case where $\theta$ is the (stabilized) tangential $2$-type of a high-dimensional manifold $M$, let us consider at the case $B=B\Spin({d+1})\times BG$. The map $\theta\colon B\Spin({d+1})\times BG\to B\ort({d+1})$ factors through the $3$-connected cover $\theta_\Spin\colon B\Spin({d+1})\to B\ort({d+1})$ and we get
\[\bun(TM\oplus\underline\bbR^2,\theta^*U_{d+1}) = \bun(TM\oplus\underline\bbR^2,\theta_\Spin^*U_{d+1})\times \map(M,BG).\]
So, a $\theta$-structure $\hat l$ on $M$ is given by a $\Spin$-structure $\sigma$ on $M$ and a map $M\to BG$. Let $\psi\coloneqq[f,L]\in B^\theta(M,\hat l)$. Then $f$ is an orientation preserving diffeomorphism of $M$ and $L$ is the homotopy class of a path connecting the bundle maps $\hat l_\Spin,\hat l_\Spin\circ df\colon TM\oplus\underline\bbR^2 \to \theta_\Spin^*U_{d+1}$ together with the homotopy class of a path connecting the maps $\alpha$ and $\alpha\circ f\colon M\to BG$. If $G=\pi_1(M,x)$ for some base-point $x\in M$, this means that the induced map $f_*\colon\pi_1(M,x)\to\pi_1(M,f(x))$ is given by conjugation by a path $\gamma\colon [0,1]\to M$ with $\gamma(0)=x$ and $\gamma(1)=f(x)$. We say that \emph{$f$ acts on the fundamental group by an inner automorphism} in this case.
\end{example}

\subsection{Cobordism sets}\label{sec:cobordism-groups}
	Before we can derive the general version of \pref{Theorem}{thm:a} we need to take a closer look at the morphism sets of $\Omega_{d,2}$. Recall that we fixed a tangential structure $\theta\colon B\to B\ort(d+1)$.

	\begin{definition}\label{def:structured-bordism-set}		
		Let $M_0^{d-1},M_1^{d-1}$ be compact manifolds with (stabilized) $\theta$-structures $\hat l_0,\hat l_1$. We define the \emph{cobordism set of manifolds with $\theta$-structure and fixed boundary}  by 
		\[\Omega^\theta_d\bigl((M_0,\hat l_0),(M_1,\hat l_1)\bigr):=\bigr\{(W,\psi_0,\psi_1,\hat\ell)\bigl\}/\sim.\]
		Here, $W$ is a $d$-manifold with boundary $\partial W =\partial_0W\ \coprod \partial_1W$, $\hat\ell\in \bun(TW\oplus\underline\bbR,\theta^*U_{d+1})$ is a bundle map and $\psi_i\colon \partial_iW\to M_i$, $i=0,1$ are diffeomorphisms such that $(-1)^i\hat l_i\circ \mathrm{d}\psi_i = \hat\ell|_{\partial_iW}$. We call $M_0$ the \emph{incoming boundary} and $M_1$ the \emph{outgoing boundary}  (see \pref{Figure}{fig:bordism-set} in \pref{Definition}{def:omega}). The equivalence relation is given by the cobordism relation: We say that $(W,\psi_0,\psi_1,\ell)$ and $(W',\psi_0',\psi_1',\ell')$ are cobordant if there exists a $(d+1)$-dimensional $\theta$-manifold $(X,\ell_X)$ with corners such that there exists a partition of $\partial X=\bigcup_{i=0,3} \partial_iX$ together with diffeomorphisms 
		\begin{align*}
			M_0\times I&\cong \partial_0 X &M_1\times I&\cong \partial_3X\\
			W &\cong \partial_2X& W'&\cong\partial_1X
		\end{align*}  
		such that $\theta$-structures and identifying diffeomorphisms fit together (see \pref{Figure}{fig:bordism-set-relation} in \pref{Definition}{def:omega}).
	\end{definition}

	\begin{remark}
		\begin{enumerate}
			\item Since $\theta$ is a fibration we do not need to consider homotopies $(-1)^i\hat l_i\circ \mathrm{df} \sim \hat\ell|_{\partial_iW}$ but we can arrange the $\theta$-structures to actually agree.
			\item $\Omega_d^\theta((M,\hat l),(M,\hat l))$ becomes a monoid via concatenation of cobordisms and identifying them along the boundary diffeomorphisms, \ie
				\[(W',\psi_0',\psi_1',\ell') \cdot (W,\psi_0,\psi_1,\ell) := (W\cup_{\psi_0'\circ \psi_1^{-1}} W', \psi_0,\psi_1', \ell\cup_{\hat \psi'_0\circ\hat \psi_1^{-1}}\ell').\]
				This monoid is actually an abelian group (cf. \pref{Corollary}{cor:isomorphism-of-bordism-groups}). 			\item Note that one has a map $\Omega_d^\theta((M,\hat l),(M,\hat l))\to \Omega_d^\theta(\emptyset,\emptyset)\eqqcolon \Omega_d^\theta$ given by identifying the \emph{equal} boundaries of a cobordism $W\colon M\leadsto M$. This map turns out to be an isomorphism of groups (cf. \pref{Corollary}{cor:isomorphism-of-bordism-groups} and the remark below it).
		\end{enumerate}
	\end{remark}
	
	\begin{prop}\label{prop:double-is-nullbordant}
		Let $W^d\colon M_0 \leadsto M_1$ be $\theta$-cobordism. Then there exists a $\theta$-structure on $W^\op\colon M_1\leadsto M_0$ such that $W\cup W^\op\sim M_0\times [0,1]\ \mathrm{rel}\ M_0\times\{0,1\}$. In particular, if $W\colon \emptyset\leadsto M$ is a $\theta$-nullbordism, the double $dW\coloneqq W\cup W^\op$ is nullbordant and $W^\op\amalg W$ is cobordant to the cylinder on $M$.
	\end{prop}
	\begin{proof}
		Consider the manifold with corners $W\times I$. We introduce new corners as in \pref{Figure}{fig:corners}.
		\begin{figure}[ht]
		\centering
		\includegraphics[width=0.8\textwidth]{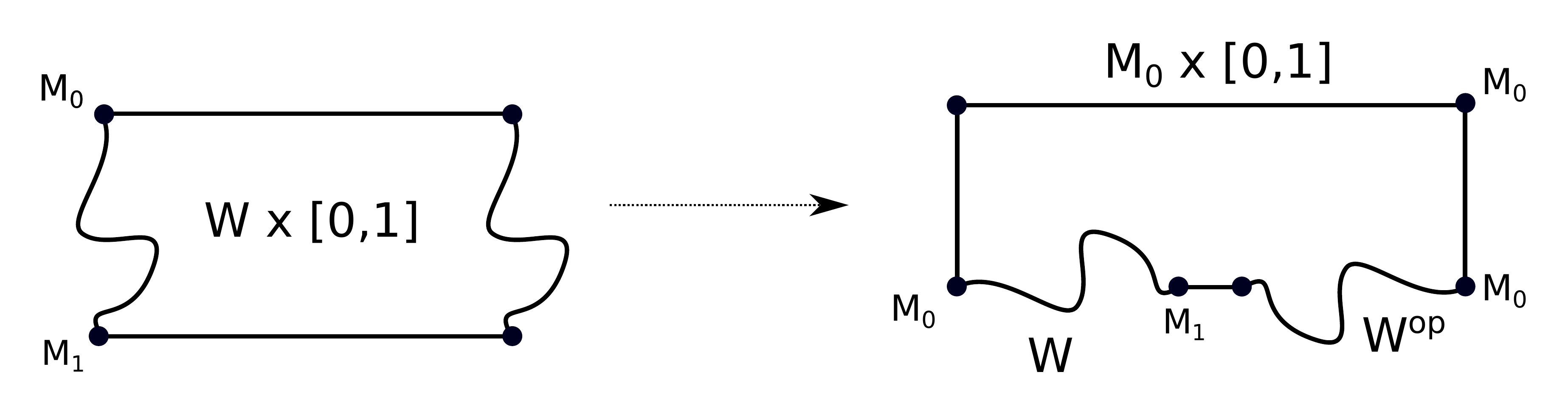}	
		\caption{Introducing corners to obtain the desired cobordism}\label{fig:corners}
		\end{figure}
		Next, we construct the $\theta$-structures\footnote{This is adapted from \cite[Proof of Proposition 2.16]{grw_stable}.}. Let $\ell_W\colon TW\oplus\underline\bbR\to\theta^*U_{d+1}$ be the $\theta$-structure on $W$. Since $W\embeds W\times[0,1]$ is a homotopy equivalence there is a unique extension up to homotopy
		\begin{center}
		\begin{tikzpicture}
			\node (0) at (0,1) {$TW\oplus\underline\bbR$};
			\node (1) at (4,1) {$\theta^*U_{d+1}$};		
			\node (2) at (0,0) {$T(W\times [0,1])$};
			
			\draw[->] (0) to node[above]{$\ell_W$} (1);
			\draw[->] (0) to (2);
			\draw[->, dotted] (2) to (1); 
		\end{tikzpicture}
		\end{center}
		where the vertical map sends $v\in\underline\bbR_{>0}$ to the inwards pointing vector. This gives a $\theta$-structure on $W\times I$ and by restriction a $\theta$-structure on $W^{\op}$.
	\end{proof}
	
	\noindent Now we can prove another useful tool.
			
	\begin{prop}\label{prop:action-of-bordism}
		The action of $\Omega_d^\theta$ on $\Omega_d^\theta\bigl((M_0,\hat l_0),(M_1,\hat l_1)\bigr)$ given by disjoint union is free and transitive.
	\end{prop}

	\begin{proof}
		Since disjoint union is associative up to cobordism and disjoint union with the emptyset is the identity and this really defines a group action. If $\Omega_d^\theta\bigl((M_0,l_0),(M_1,l_1)\bigr)=\emptyset$ the statement is trivial. So let $(L,\psi_0^L,\psi_1^L,\ell_L)\colon (M_0,\hat l_0)\leadsto (M_1,\hat l_1)$ be a cobordism. Let $\Phi_L\colon\Omega_d^\theta\longrightarrow\Omega_d^\theta\bigr((M_0,l_0),(M_1,l_1)\bigl)$ be given by $\Phi_L(V) = V\amalg L$. Also let 
		\[\tilde \Phi_L\colon \Omega_d^\theta\bigr((M_0,l_0),(M_1,l_1)\bigl)\too \Omega_d^\theta\]
		be defined given by gluing in the cobordism $(L^{\op},\psi_1^L,\psi_0^L,\ell_L^{\op})$ along the boundary as follows: We concatenate with $L^\op$ and then identify the equal boundaries:
			\[\Omega_d^\theta\bigr((M_0,l_0),(M_1,l_1)\bigl) \overset{\cup L^\op}\too\Omega_d^\theta\bigr((M_0,l_0),(M_0,l_0)\bigl)\too\Omega_d^\theta\]
		We will prove the Proposition by showing that $\Phi$ and $\tilde \Phi$ are mutually inverse bijections. The easy part is 
		\[\tilde\Phi(\Phi([V])) = \tilde\Phi([V\amalg L]) = [V\amalg (L\cup L^{\op})] = [V]\]
		by \pref{Proposition}{prop:double-is-nullbordant}. It remains to show that $\tilde\Phi_L(W)\amalg L = (W\cup L^{\op}) \amalg L$ is cobordant to $W$. First we note that $(W,\psi_0,\psi_1)$ is diffeomorphic to $(M_0\times I\cup_{\psi_0}W\cup_{\psi_1^{-1}}M_1\times I,\id,\id)$ and so it suffices to consider the case that all boundary identifications are given by the identity. We now decompose $(W\cup L^{\op})\amalg L$ as follows:
		\begin{align*}
			V_0&\coloneqq M_0\times[0,\epsilon]\cup M_1\times[1-\epsilon,1] &V_1&\coloneqq L\\
			V_2&\coloneqq L^{\op} &V_3&\coloneqq W 
		\end{align*}
		Note that 
		\begin{align*}
			\partial V_0 &= (M_0\times\{0\})\amalg \underbrace{(M_0\times \{\epsilon\})\amalg (M_1\times\{1-\epsilon\})}_{\eqqcolon \partial_+V_0} \amalg (M_1\times \{1\})\\
			\partial V_1 &= M_0\amalg M_1 = \partial V_2 = \partial V_3 
		\end{align*}
		By identifying $\partial_+ V_0$ and $\partial V_2$ with $\partial V_1$ and $\partial V_3$ in different ways we obtain
		\begin{align*}
			V_0\cup V_1 & = L &V_2\cup V_3& = L^\op\cup W\\
			V_0\cup V_3 & = W &V_2\cup V_3& = L^\op\cup L=dL
		\end{align*}
		We will now construct the cobordism $X\colon(V_0\cup V_1)\amalg (V_2\cup V_3)\leadsto(V_0\cup V_3)\amalg (V_2\cup V_1)$. We construct this by taking $V_i\times I$ for every $i=0,1,2,3$, introducing corners at the boundary (and at $\partial_+V)$ respectively) as shown in \pref{Figure}{fig:introducing-corners}.
		\begin{figure}[ht]
		\centering
		\includegraphics[width=0.5\textwidth]{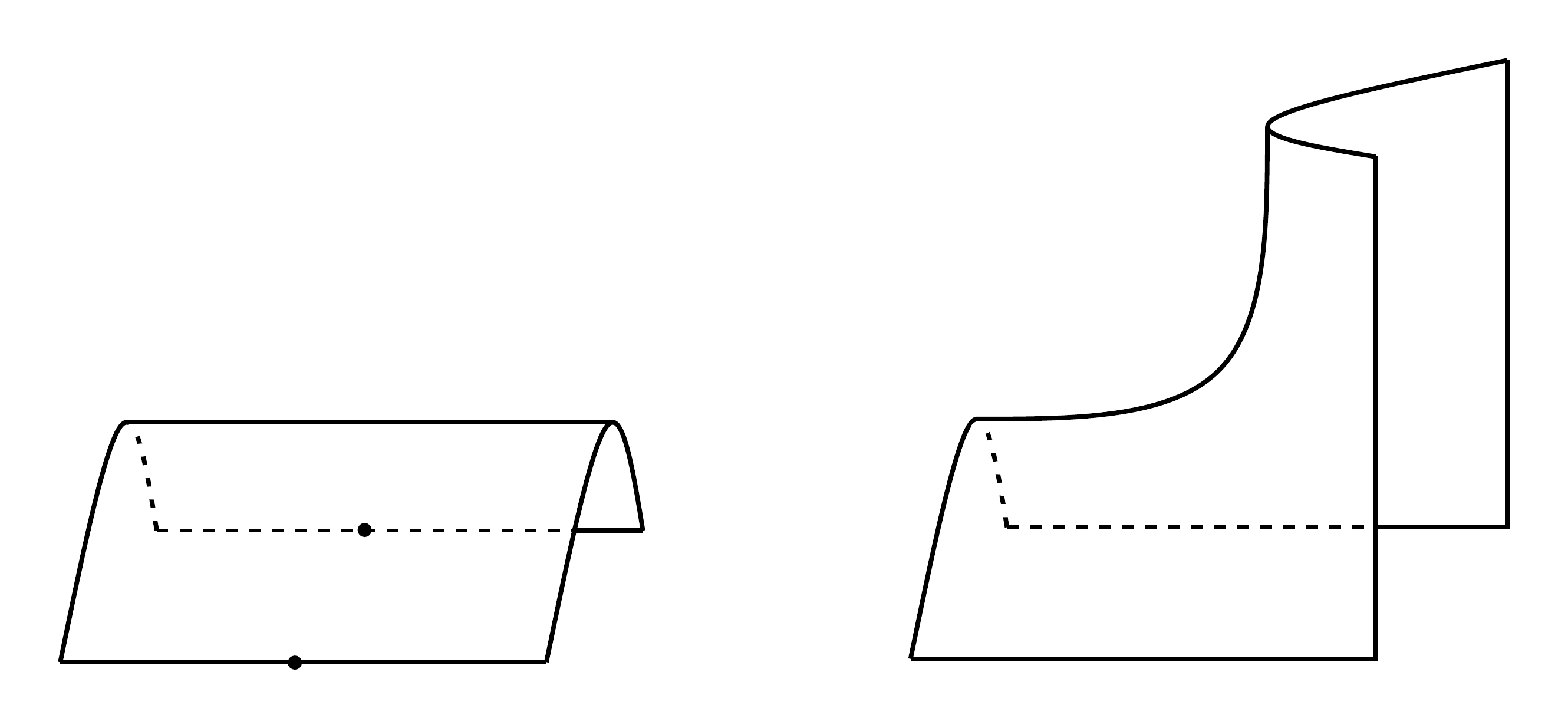}
		\caption{Introducing corners at the boundary of $V_i\times [0,1]$}\label{fig:introducing-corners}
		\end{figure}\\
		We then glue together these manifolds as follows: We identify 
		\begin{itemize}
			\item$\partial V_0\times\{t\} \text{ with } \partial V_1\times \{t\} \text{ and } \partial V_2\times\{t\} \text{ with } \partial V_3\times \{t\} \text{ for } t\in[0,\frac12]$
			\item$\partial V_0\times\{t\} \text{ with } \partial V_3\times \{t\} \text{ and } \partial V_1\times\{t\} \text{ with } \partial V_2\times \{t\} \text{ for } t\in[\frac12,1]$.
		\end{itemize}
		This is shown in \pref{Figure}{fig:cutting-and-pasting}. The $\theta$-structures are given by $\hat \ell_{V_i}\oplus\id_{\underline\bbR}$ (the arrows in \pref{Figure}{fig:cutting-and-pasting} indicate the incoming and outgoing boundary of $X$).
		\begin{figure}[ht]
		\centering
		\includegraphics[width=0.5\textwidth]{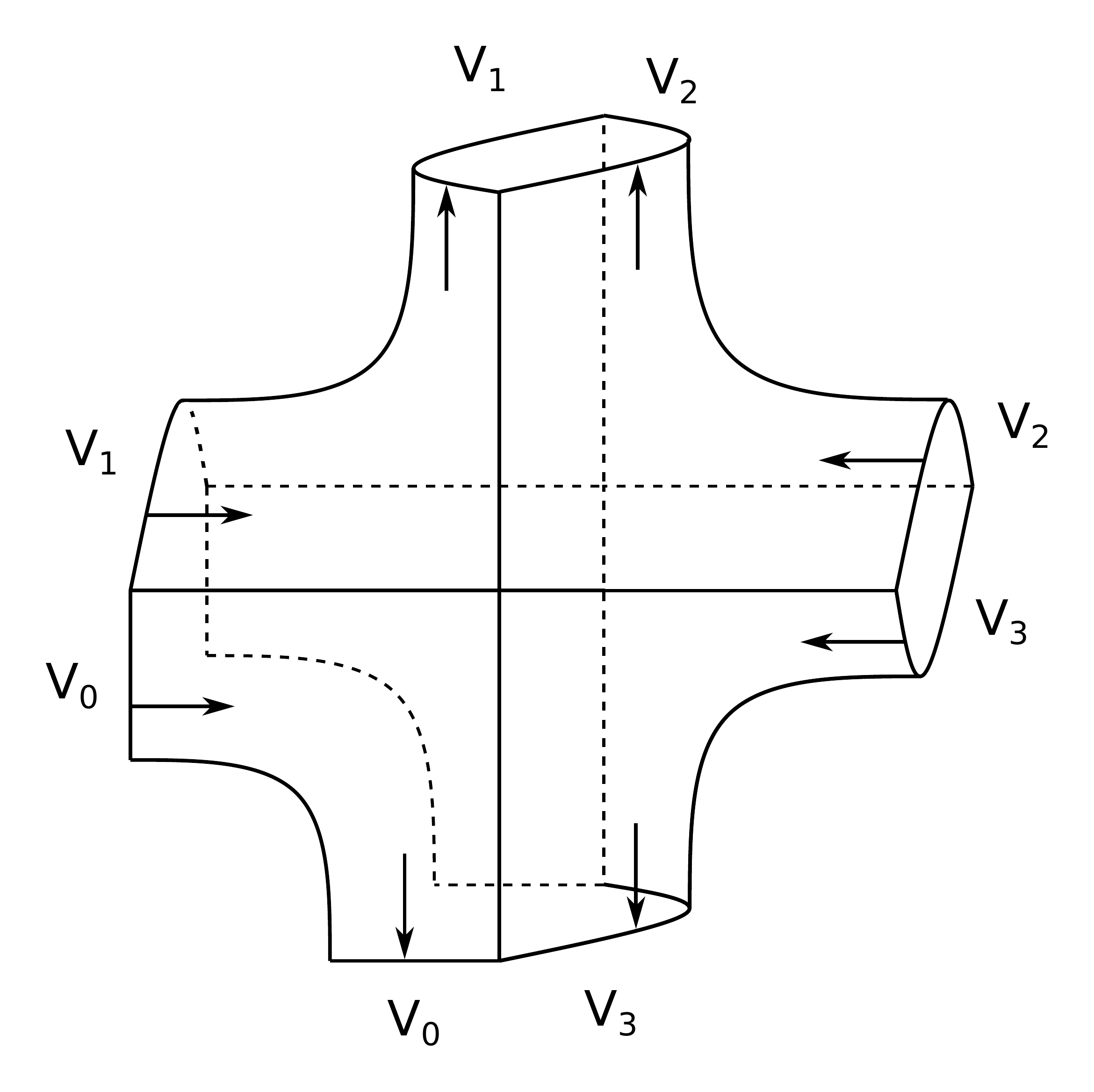}
		\caption{The cobordism $X\colon (V_0\cup V_1)\amalg (V_2\cup V_3)\leadsto (V_0\cup V_3)\amalg (V_2\cup V_1)$}\label{fig:cutting-and-pasting}
		\end{figure}
	\end{proof}
	\begin{remark}\label{rem:action-of-bordism}
		\pref{Proposition}{prop:action-of-bordism} can also be proven using structured cobordism categories. The presented proof however is much more direct.
	\end{remark}

	\begin{cor}\label{cor:isomorphism-of-bordism-groups}
		Let $(M,l)$ be a $(d-1)$-dimensional $\theta$-manifold. Then the map
		\[\Phi\colon\Omega^\theta_d\congarrow\Omega^{\theta}_d\bigr((M,\hat l),(M,\hat l)\bigr)\]
		given by $(V,\hat\ell)\mapsto (M\times[0,1]\amalg V,\id,\id ,(\hat l\oplus\id_{\underline\bbR})\amalg\hat\ell)$ is an isomorphism of groups. In particular, $\Omega^{\theta}_d\bigr((M,\hat l),(M,\hat l)\bigr)$ is an abelian group.
	\end{cor}
	
	\begin{proof}
		It is a group homomorphism because
		\begin{align*}
			\Phi(V\amalg W) &= M\times [0,1]\ \amalg\  V\amalg W\\&=(M\times[0,1]\ \amalg\  V)\cup (M\times [1,2]\ \amalg\  W) = \Phi(V)\cup\Phi(W).
		\end{align*}
		The rest follows from \pref{Proposition}{prop:action-of-bordism}.
	\end{proof}
	
	\begin{remark}
		The inverse is given by mapping $(W,\psi_0,\psi_1)$ to the manifold obtained by gluing $\partial_1W$ to $\partial_0W$ along the diffeomorphism $\psi_0^{-1}\circ\psi_1$. 
	\end{remark}
	
	\begin{cor}\label{cor:mapping-torus-homomorphism}
		The map $\Gamma^\theta(M,\hat l)\to \Omega_d^{\theta}$ given by $[\gamma]\mapsto [M_\gamma]$ is a homomorphism.
	\end{cor}
	
	\begin{proof}
		Let $\gamma\colon[0,1]\to B\diff^\theta(M)$ be a path from $(M,\hat l)$ to itself. We define the \emph{mapping cylinder map} by $A\colon\Gamma^\theta(M,\hat l)\to \Omega_d^{\theta}(M,M), \gamma\mapsto (\gamma^*U_{M,\theta},\id,\id)$. Since the bundle classified by $\gamma_0*\gamma_1$ is the same as the union of the bundles classified by $\gamma_i$, this satisfies 
		\begin{align*}
			A(\gamma_0*\gamma_1)& = ((\gamma_0*\gamma_1)^*U_{M,\theta},\id,\id)\\
				&= (\gamma_0^*U_{M,\theta} \cup \gamma_1^*U_{M,\theta},\id,\id)\\
				&= (\gamma_0^*U_{M,\theta},\id,\id)\cup(\gamma_1^*U_{M,\theta},\id,\id) = A(\gamma_0)\cup A(\gamma_1).
		\end{align*}
		Since the isomorphism $\Omega_d^{\theta}(M,M)\to \Omega_d^\theta$ is given by gluing the boundary, we have $M_\gamma = \tilde\Phi(\gamma^*U_{M,\theta})$ and hence
		\begin{align*}
		M_{\gamma_0*\gamma_1} &= \tilde\Phi(A(\gamma_0*\gamma_1))=\tilde\Phi(A(\gamma_0)\cup A(\gamma_1)) \\
			&\overset{\ref{cor:mapping-torus-homomorphism}}= \tilde\Phi(A(\gamma_0))\amalg\tilde\Phi(A(\gamma_1)) = M_{\gamma_0}\amalg M_{\gamma_1}. \qedhere
		\end{align*}
	\end{proof}
	
	\begin{remark}\label{rem:mapping-torus-homomorphism}
		Using the model $B^\theta(M,\hat l)$ for the mapping class group, we see that the map $A\colon B^\theta(M,\hat l)\to \Omega_d^{\theta}(M,M)$ is given by $\psi\mapsto (M\times[0,1],\id,\psi^{-1})$ for $P$ . Note that since $\Omega_d^\theta(M,M)$ is commutative, $\psi\mapsto (M\times[0,1],\id,\psi)$ is a homomorphism as well.
\end{remark}

\subsection{The action of the mapping class group}\label{chap:applications}

\noindent We will now give the general statement of \pref{Theorem}{thm:a}. For a space $X$ let $\haut(X)$ denote the group-like $H$-space of weak homotopy equivalences of $X$. 

\begin{cor}\label{cor:action-of-mcg}
	Let $d\ge7$, let $M$ be a $(d-1)$-dimensional manifold and let $\theta\colon B\to B\ort(d+1)$ be the stabilized tangential $2$-type of $M$ where $\hat l\colon TM\oplus\underline\bbR^2\to\theta^*U_{d+1}$ is a $\theta$-structure. Then there exists a group homomorphism	
		\[\se\colon\Omega_d^{\theta}\too \pi_0(\haut(\calR^+(M))),\]
	such that the following diagram, where $F$ is the forgetful map and $T$ is the mapping torus map, commutes
	\begin{center}
	\begin{tikzpicture}
		\node (0) at (0,1.5) {$\Gamma^\theta(M,\hat l)$};
		\node (1) at (0,0) {$\Gamma(M)$};
		\node (2) at (5,1.5) {$\Omega_d^\theta$};
		\node (3) at (5,0) {$\pi_0(\haut(\calR^+(M))).$};
		
		\draw[->] (0) to node[auto]{$F$} (1);
		\draw[->] (0) to node[auto]{$T$} (2);
		\draw[->] (1) to node[auto]{$\Theta$} (3);
		\draw[->] (2) to node[auto]{$\se$} (3);
	\end{tikzpicture}
	\end{center}
\end{cor}

\begin{proof}
	We define $\se(W)\coloneqq \calS_{M\times[0,1]\amalg W,\id,\id}$. Then $\se$ is a group homomorphism by \pref{Corollary}{cor:isomorphism-of-bordism-groups} and because $\calS$ is compatible with composition. By \pref{Theorem}{thm:main} the above diagram is commutative since $[M\times I\ \amalg\  T_\psi,\id,\id] = [M\times I,\id,\psi^{-1}]$ (cf. \pref{Corollary}{cor:mapping-torus-homomorphism} and \pref{Remark}{rem:mapping-torus-homomorphism}).
\end{proof}

\begin{proof}[Proof of \pref{Theorem}{thm:actionofmcg}]
	Since $M$ ist simply connected and stably parallelizable, the tangential $2$-type of $M$ is given by $B\Spin(d+1)$. Let $f\colon M\to M$ be a $\Spin$-diffeomorphism. By \pref{Corollary}{cor:action-of-mcg} we need to show that the class $[T_f]\in\Omega_d^\Spin$ vanishes. The kernel of the forgetful homomorphism $\Omega^\Spin_d\to\Omega^{\SO}_d$ is a finite dimensional $\bbF_2$-vector space and concentrated in degrees congruent $\equiv1,2(8)$ (cf. \cite{andersonbrownpeterson_short}). By \cite[Proposition 13]{kreck_bordism_odd} mapping tori of stably parallelizable manifolds are orientedly nullbordant which finishes the proof.
\end{proof}

\noindent There are more examples to which \pref{Corollary}{cor:action-of-mcg} is applicable which can be found in \cite[Section 4.1]{ownthesis}. 

\begin{proof}[Proof of \pref{Proposition}{prop:spheres}]
	By \pref{Proposition}{prop:spin-mcg} we may assume that $f$ is a $\Spin$-diffeomorphism. Let $W\colon S^{d-1}\to S^{d-1}$ be an admissible cobordism $\Spin$-cobordant to $S^{d-1}\times[0,1]\amalg T_f$. Then $f^*\sim\se_{T_f}\sim\calS_W$ and by \pref{Proposition}{prop:single-metric} and \pref{Remark}{rem:single-metric} this is homotopic to the identity if $\calS_W(g_\circ)$ is homotopic to $g_\circ$.
\end{proof}

\subsection{The action for simply connected $\Spin$ $7$-manifolds}\label{sec:triviality-criterion}

\noindent We have the following result for $7$-manifolds which implies \pref{Corollary}{cor:a}.

\begin{cor}\label{cor:7-dimensional}
	Let $M^7$ be a simply connected $\Spin$-manifold and let $f\colon M\congarrow M$ be a $\Spin$-diffeomorphism. Then the following are equivalent:
	\begin{enumerate}		
		\item $\hat\scrA(T_f)=0$.
		\item $T_f$ is $\Spin$-nullbordant.
		\item $f^*$ is homotopic to the identity.
		\item $f^*g\sim g$ for every $g\in\calR^+(M)$.
		\item There exists a metric $g\in\calR^+(M)$ such that $f^*g\sim g$.
	\end{enumerate}
\end{cor}

\begin{proof}
	The implications $3.\Rightarrow 4.$ and $4.\Rightarrow 5.$ are obvious and the implication $2.\Rightarrow 3$ follows from \pref{Corollary}{cor:action-of-mcg}. For $1.\Rightarrow 2.$ we note that
	\[\Omega^{\Spin}_8\cong \bbZ\oplus\bbZ \cong\langle [\hp2],[\beta]\rangle,\]
	where $\beta$ denotes the Bott manifold with $\hat\scrA(\beta)=1$ and $\sign(\beta)=0$. Furthermore, $\sign(\hp2)\ne0$ and $\hat\scrA(\hp2)=0$. Since for $T_f$ both these invariants vanish, it has to be $\Spin$-nullbordant. Finally $5.\Rightarrow 1.$ is proven as follows: Let $g_t$ be an isotopy between $f^*g$ and $g$. Since isotopy of psc-metrics implies concordance of psc-metrics, there exists a psc-metric $G$ on $M\times[0,1]$ restricting to $f^*g$ and $g$. Then $G$ induces a psc-metric on $T_f$ as one can identify the metrics on the boundary along $f^*$ and hence $\hat \scrA(T_f)=0$. 
\end{proof}

\begin{remark}
	Since $M$ is simply connected we have $\diff^{\Spin}(M)\twoheadrightarrow\diff^+(M)$. Hence the above Corollary classifies the action of $\Gamma^+(M)$ on $\calR^+(M)$ for every simply connected $7$-dimensional $\Spin$-manifold.	
\end{remark}

\noindent Note that the implications $2.\Rightarrow 3.\Rightarrow 4.\Rightarrow 5.\Rightarrow 1.$ don't require the restriction to dimension $7$.\begin{samepage} For simply connected, $7$-dimensional $\Spin$-manifolds, we get a further factorization of the action map:

\begin{center}
	\begin{tikzpicture}
		\node (0) at (0,3) {$\Gamma^{\Spin}(M)$};
		\node (1) at (3,2) {$KO^{-8}(\pt)$};
		\node (2) at (6,3) {$\pi_0(\haut(\calR^+(M)))$};
		\node (5) at (3.8,1.5) {$\hat\scrA(\beta)$};
		\node (6) at (6.5,2.5) {$\se_\beta$};

		\draw[->] (0) to node[above]{$\eta$} (2);
		\draw[->] (1) to (2);
		\draw[|->] (5) to (6);
		\draw[->] (0) to node[below, xshift=-15pt, yshift=7pt]{$\hat\scrA\circ T$} (1);
	\end{tikzpicture}
\end{center}
\end{samepage}
\noindent This factorization is sharp in the sense that $\ker\eta=\ker\hat\scrA\circ T$ by \pref{Corollary}{cor:7-dimensional}.

\begin{prop}\label{prop:a-hat-genus}
	Let $M$ be a $(d-1)$-dimensional, simply connected $\Spin$-manifold and let $W^d$ be a closed $\Spin$-manifold with $\alpha(W)\ne0$. Then $\se_W(g)\not\sim g$ for every psc-metric $g$ on $M$.
\end{prop}

\begin{proof}
	By \pref{Lemma}{lem:connectivity-of-bordisms} we can perform ($\Spin$-)surgery on the interior of $M\times[0,1]\amalg W$ to get an admissible cobordism $V\colon M\leadsto M$. If there exists a psc-metric $g_0\in\calR^+(M)$ such that $\se_W(g_0)\sim g_0$, there exists a psc-metric $G$ on $V$ that restricts to $g_0$ on both boundaries by \pref{Remark}{rem:walsh}. We obtain a psc-metric on the manifold $V'$ given by gluing the boundaries of $V$ together along the identity. So, $\alpha(V')=0$ by the Lichnerowicz-formula and since $\alpha$ is $\Spin$-cobordism invariant we get 
	\begin{align*}
		0&=\alpha(V')=\alpha((M\times S^1)\amalg W) = \alpha(W).\qedhere
	\end{align*}
\end{proof}

\noindent This shows that vanishing of the $\alpha$-invariant of $W$ is necessary condition for $\se(W)$ to be homotopic to the identity. We close with the following question.

\begin{que}\label{question1}
	Let $M$ be simply connected and $\Spin$. Is vanishing of the $\alpha$-invariant of $W$ a sufficient condition for $\se(W)$ to be homotopic to the identity on $\calR^+(M)$?
\end{que}


\noindent If the answer to \pref{Question}{question1} were yes, we would get the following diagram. 
\begin{center}
	\begin{tikzpicture}
		\node (0) at (0,2) {$\Gamma^{\Spin}(M)$};
		\node (1) at (1.5,1) {$\Omega^{\Spin}_d$};
		\node (2) at (6,2) {$\pi_0(\haut(\calR^+(M)))$};
		\node (8) at (3,0) {$KO^{-d}(\pt)$};
		
		\draw[->] (0) to (2);
		\draw[->] (1) to (2);
		\draw[->] (0) to (1);
		\draw[->] (1) to (8);
		\draw[->] (8) to (2);
	\end{tikzpicture}
\end{center}

\bibliographystyle{halpha-abbrv}
\bibliography{Bibliography}

\end{document}